\documentclass[12pt]{article}

\usepackage[applemac]{inputenc}
\usepackage{amsmath}
\usepackage{amssymb}
\usepackage{amsthm}
\usepackage[a4paper]{geometry}
\usepackage[applemac]{inputenc}
\usepackage{bbm}
\usepackage{graphicx}
\usepackage{xcolor}

\font\tenBbb=msbm10
\font\sevenBbb=msbm7
\font\fiveBbb=msbm5
\newfam\Bbbfam
\textfont\Bbbfam=\tenBbb
\scriptfont\Bbbfam=\sevenBbb
\scriptscriptfont\Bbbfam=\fiveBbb
\def\Bbb{\fam\Bbbfam\tenBbb}

  \newtheorem{theorem}{Theorem}[section]
  \newtheorem{lemma}[theorem]{Lem
  ma}
  
  \newtheorem{proposition}[theorem]{Proposition}
  \newtheorem{definition}[theorem]{Definition}

  \newcounter{figures}[section]

\newcommand{\be}{\begin{equation}}
\newcommand{\ee}{\end{equation}}
\newcommand{\bea}{\[\begin{array}{lll}}
\newcommand{\eea}{\end{array}\]}

\newcommand{\bi}{\begin{itemize}}
\newcommand{\ei}{\end{itemize}}
\newcommand{\bn}{\begin{enumerate}}
\newcommand{\en}{\end{enumerate}}
\newcommand{\bl}{\begin{lemma}}
\newcommand{\el}{\end{lemma}}
\newcommand{\bth}{\begin{theorem}}

\newcommand{\bE}{{\Bbb E}}

\newcommand{\bH}{{\Bbb H}}
\newcommand{\bK}{{\Bbb K}}
\newcommand{\bL}{{\Bbb L}}

\newcommand{\bR}{{\Bbb R}}
\newcommand{\bS}{{\Bbb S}}

\newcommand{\cP}{{\cal P}}
\newcommand{\cH}{{\cal H}}

\newcommand{\cN}{{\cal N}}

\newcommand{\cV}{{\cal V}}
\newcommand{\cM}{{\cal M}}

\newenvironment{rem}{\textit{Remark} : }{\hfill$\star$\newline }
\newenvironment{rems}{{\it Remarks} : \bi}{\hfill$\star$ \ei }

\newcommand{\Arccos}{\mathop{\textrm{Arccos
  }}}

\def\eps{\varepsilon}
\def\mod{{\rm mod \,}}
\def\a{\alpha}
\def\b{\beta}
\def\supp{{\rm supp \,}}
\def\ONE{{\mathbbm 1}}
\def\tONE{{\tilde{\mathbbm 1}}}

\newcommand{\lambdaquadra}{\omega}

\author{Kerkyacharian G\'erard, Kyriazis George, Le Pennec Erwan,\\
Petrushev Pencho and Picard  Dominique}

\date{March 11, 2009}

\title{Inversion of noisy Radon transform by SVD based needlets}

\begin{document}

\maketitle

\begin{abstract}
A linear method for inverting noisy observations of the Radon transform is developed
based on decomposition systems (needlets) with rapidly decaying elements
induced by the Radon transform SVD basis. Upper bounds of the risk of the estimator
are established in $L^p$ ($1\le p\le \infty$) norms for functions with Besov space smoothness.
A practical implementation of the method is given and several examples are discussed.
\end{abstract}

\section{Introduction}
\setcounter{equation}{0}

Reconstructing images (functions) from their Radon transforms is a fundamental problem in
medical imaging and more generally in tomography.
The problem is to find an accurate and efficient algorithm for approximation of the function to be recovered
from its Radon projections.
In this paper, we consider the problem of inverting noisy observations of the Radon transform.
As in many other inverse problems, there exists a basis which is fully adapted to the problem,
in particular, the inversion in this basis is very stable;
this is the Singular Value Decomposition (SVD) basis.
The Radon transform SVD basis, however, is not quite suitable for decomposition of functions
with regularities in other than $L^2$-related spaces.
In particular, the SVD basis is not quite capable of representing local features of images,
which are especially important to recover.

 The problem requires a special construction
adapted to the sphere and the Radon SVD, since usual tensorized wavelets will never reflect the
manifold structure of the sphere and will necessarily create unwanted artifacts, or will concentrate on special features (such as ridgelets...).

Our idea is to design an estimation method for inverting the Radon transform which
has the advantages of maximum localization of wavelet based methods combined with
the stability and computability of the SVD methods.
To this end we utilize the construction from \cite{pxuball} (see also \cite{pxukball})
of localized frames based on orthogonal polynomials on the ball, which are closely related to
the Radon transform SVD basis.
As shown in the simulation section the results obtained are quite promising.

To investigate the properties of this method, we perform two different studies. The first study is of theoretical kind and investigates the possible losses (in expectation) of the method in the 'minimax framework'. This principle, fairly standard in statistics, consists in analyzing the mathematical properties of estimation algorithms via optimization of their worst case performances over large ensembles of parameters. We carry out this study in a random model which is also well known in statistics, the white noise model.
This random model is a toy model well admitted in statistics since the 80's as an approximation of the 'real' model on scattered data. It is proved, for instance in \cite{brownlow96} that the regression model with uniform design and the white noise model are close in the sense of Le Cam's deficiency -which roughly means that any procedure can be transferred from one model to the other, with the same order of risk-. 
This model has the main advantage of avoiding unnecessary technicalities.
In this context we prove that over large classes of functions
(described later), our method has  optimal rates of convergence, for
all the $L^p$ losses. To our knowledge,  most of the  parallel results
are generally stated for $L^2$ losses, as in
\cite{madych83:_polyn_based_algor_comput_tomog} for example, very few (if any) consider  $L^p$ losses while it is a warrant for instance that the procedure will be able to detect small features. 
Again, the problem of choosing appropriated spaces of regularity in this context
is a serious question, and it is important to consider the spaces which
may be the closest to our natural intuition: those which generalize to the
present case the approximation properties shared by standard Besov and Sobolev spaces. We can also prove that our results apply for ordinary Besov spaces.

In the case $p\ge 4$ we exhibit here new minimax rates of convergence,
related to the ill posedness coefficient of the inverse problem $\frac{d-1}2$
along with edge effects induced by the geometry of the ball.
These rates are interesting from a statistical point of view and  have to be compared with similar phenomena occurring in other inverse problems
involving Jacobi polynomials (e.g. Wicksell problem), see \cite{kppw}.

Our second study of the performances of our procedure is performed on simulations. Since in practical situations  scattered data are generally observed,  we  carried out our  simulation study in the scattered data model. 
We basically compared our method  to the SVD procedure -since it is the most commonly studied method in statistics- and the simulation study
 consistently predicts  quite good performances of our
procedure and a comparison  extensively
 in favor of our algorithm.
One could object that it is a rather common opinion  that 'one should smooth
the SVD'.
  However,  there are many ways to do
so (for instance, we mention a parallel method, employing a similar idea for smoothing out the projection operator
but without using the needlet construction and in a no-noise framework, which
has been developed by Yuan Xu and his co-authors in \cite{xu7, XUTI, xuti1}). Ours has the advantage of being optimal for at least one point
of view since we are able to obtain the right rates of convergence in
$L_p$ norms.

The paper is structured as follows.
In Section 2 we introduce the model and the Radon transform Singular Value Decomposition.
In Section 3 we give the class of linear estimators built upon the SVD.
We also give the needlet construction and introduce the needlet estimation algorithm.
In Section 4 we
establish bounds for the risk of this estimate over large classes of regularity spaces.
Section 5 is devoted to the practical implementation and results of our method.
Section 6 is an appendix where  the proofs of some claims from Section 3 are given.

\section{Radon transform and white noise model}
\setcounter{equation}{0}

\subsection{Radon transform}

Here we recall the definition and some basic facts about the Radon transform
(cf. \cite{Helgason}, \cite{NATT}, \cite{LOG}).
Denote by $B^d$ the unit ball in $\bR^d$,
i.e.
$B^d = \{x  =(x_1,\ldots,x_d)\in \bR^d: |x| \le 1\}$ with $|x|=(\sum_{i=1}^d x_i^2)^{1/2}$
and by $\bS^{d-1}$ the unit sphere in $\bR^d$.
The Lebesgue measure on $B^d$ will be denoted by $dx$
and the usual surface measure on $\bS^{d-1}$ by $d\sigma(x)$
(sometimes we will also deal with the surface measure on $\bS^d$
which will be denoted by $d\sigma_d$).
We let $|A|$ denote the measure $|A| = \int_A dx$ if $A\subset B^d$
as well as  $ |A| = \int_A d\sigma(x) $ if $A\subset \bS^{d-1}$.

The Radon transform of a function $f$ is defined by
\[
Rf(\theta, s) =
\int_{\substack{y\in\theta^\perp\\
s\theta+y\in B^d}} f(s\theta  + y) dy,
\quad \theta \in \bS^{d-1}, \;s \in[-1,1],
\]
where $dy$ is the Lebesgue measure of dimension $d-1$ and
$\theta^\perp=\{x \in \bR^d:  \langle x , \theta \rangle = 0 \}$.
With a~slight abuse of notation, we will rewrite this integral as
\[
Rf(\theta, s) = \int_{\langle y,\theta \rangle
    =s}f(y) dy.
\]
It is easy to see (cf. e.g. \cite{NATT}) that the Radon transform is a bounded linear operator mapping
$\bL^2(B^d,dx)$ into $\bL^2 \left(\bS^{d-1} \times  [-1,1],  d\mu(\theta,s)\right)$,
where
\[
  d\mu(\theta, s) = d\sigma(\theta) \frac{ds}{(1-s^2)^{(d-1)/2}}.
\]

\subsection{Noisy observation of the Radon transform}

We consider observations  of the form
\begin{equation}\label{WNM}
 dY(\theta, s) = Rf (\theta, s) d\mu(\theta, s) + \eps  dW(\theta, s),
\end{equation}
where the unknown function $f$ belongs to
$\bL^2(B^d,dx) $.
The meaning of this equation is that for any
$\phi(\theta,s)$ in $\bL^2(\bS^{d-1} \times  [-1,1],  d\mu(\theta,s) )$
one can observe
\begin{align*}
Y_\phi
&=\int \phi(\theta,s ) dY(\theta,s) = \int_{\bS^{d-1} \times  [-1,1]} Rf(\theta,s)
\phi(\theta,s) d\mu(\theta,s) +  \eps \int \phi(\theta,s)dW(\theta,s)\\
&= \langle Rf, \phi \rangle_\mu + \eps W_\phi.
\end{align*}
Here $W_\phi =\int \phi(\theta,s) dW(\theta,s)$ is a Gaussian field of zero  mean and covariance
\[
\bE(W_\phi, ~W_\psi )=  \int_{\bS^{d-1} \times  [-1,1]} \phi(\theta ,s) \psi(\theta, s) d\sigma(\theta)
\frac{ds}{(1-s^2)^{(d-1)/2}} = \langle \phi, \psi \rangle_\mu.
\]
The goal is to recover the unknown function $f$ from the observation of $Y$.
As explained in the introduction, this model is a toy model, fairly accepted in statistics as an approximation of the 'real model' of scattered data. The study is carried out in this setting to avoid unnecessary technicalities.

Our idea is to devise an estimation scheme which combines
 the stability and computability of SVD decompositions with
the superb localization and multiscale structure of wavelets.
To this end we utilize a frame (essentially following the construction from \cite{pxukball})
with elements of nearly exponential localization
which is compatible with the SVD basis of the Radon transform.
This procedure is also to be considered as a first step towards a nonlinear procedure especially suitable to handle  spatial  adaptivity since real objects frequently exhibit a variety of shapes and spatial inhomogeneity.

\subsection{Polynomials and Singular Value Decomposition of the Radon transform}\label{SVD-Radon}

The SVD of the Radon transform was first established in
\cite{cormack64:_repres_ii,Davison, Louis}.
In~this regard we also refer the reader to \cite{NATT, xu7}.
In this section we record some basic facts related to the Radon SVD and recall some standard definitions which will be used in the sequel.

\subsubsection{Jacobi and Gegenbauer polynomials}\label{Jacobi}

The Radon SVD bases are defined in terms of Jacobi and Gegenbauer polynomials. \\
The Jacobi polynomials $P_n^{(\alpha,\beta)}$, $n\ge 0$,
constitute an orthogonal basis for the space
$\bL^2([-1, 1], w_{\alpha,\beta}(t)dt)$ with weight $w_{\alpha,\beta}(t)=(1-t)^\a(1+t)^\b$,
$\a, \b>-1$.
They are standardly normalized by
$P_n^{(\a,\b)}(1)={n+\a \choose n}$ and then \cite{AAR, Erdelyi, SZG}
$$
\int_{-1}^1 P_n^{(\alpha,\beta)}(t)
P_m^{(\a,\b)}(t)w_{\alpha,\beta}(t)dt  = \delta_{n, m} h_n^{(\alpha,\beta)},
$$
where
\begin{equation}\label{def-hn}
h_n^{(\alpha,\beta)} =
\frac{2^{\alpha+\beta+1}}{(2n+\alpha+\beta+1)}
\frac{\Gamma(n+\a+1)\Gamma(n+\b+1)}{\Gamma(n+1)\Gamma(n+\a+\b+1)}.
\end{equation}
The Gegenbauer polynomials $C_n^\lambda$ are a particular case of Jacobi polynomials, traditionally defined by
$$
C_n^\lambda(t)
=\frac{(2\lambda)_n}{(\lambda+1/2)_n}
P_n^{(\lambda-1/2,\,\lambda-1/2)}(t),
\quad \lambda>-1/2,
$$
where by definition $(a)_n= a(a+1)\dots(a+n-1)=\frac{\Gamma(a+n)}{\Gamma(a)}$.
It is readily seen that
$C_n^\lambda(1)= {n+2\lambda-1 \choose n}=\frac{\Gamma(n+2\lambda)}{n!\Gamma(2\lambda)}$ and
\begin{equation}\label{gegennorm}
\int_{-1}^1  C^\lambda_n (t) C^\lambda_m(t) (1-t^2)^{\lambda-\frac 12}dt
=\delta_{n, m}h_n^{(\lambda)}
\quad\mbox{with}\quad
h_n^{(\lambda)}= \frac{2^{1-2\lambda}\pi}{\Gamma(\lambda)^2}
\frac{\Gamma(n+2\lambda)}{(n+\lambda)\Gamma(n+1)}.
\end{equation}

\subsubsection{Polynomials on \boldmath $B^d$ and $\bS^{d-1}$}\label{polynom}

We detail the following well known notations which will be used in the sequel.
Let $\Pi_n(\bR^d) $ be the space of all polynomials in $d$ variables of degree $\le n$.
We denote by $\cP_n(\bR^d)$ the space of all homogeneous polynomials of degree $n$
and by  $\cV_n(\bR^d)$ the space of all polynomials of degree $n$
which are orthogonal to lower degree polynomials with respect to the Lebesgue measure on $B^d$.
$\cV_0$ is the set of  constants.
We have the following orthogonal decomposition:
$$
\Pi_n(\bR^d) =\bigoplus_{k=0}^n \cV_k(\bR^d).
$$

Also, denote by $\cH_n(\bR^d)$ the subspace of all harmonic homogeneous polynomials of degree~$n$
and by $\cH_n(\bS^{d-1})$  the 
 restriction of the polynomials from $\cH_n(\bR^d)$ to $\bS^{d-1}$.
Let $\Pi_n(\bS^{d-1}) $ be the space of  restrictions to $\bS^{d-1}$ of polynomials of degree
$\leq n$ on $\bR^d$.
As is well known
\[
\Pi_n(\bS^{d-1})= \bigoplus_{m=0}^n  \cH_{m}(\bS^{d-1})
\]
(the orthogonality is with respect of the surface measure $d\sigma$ on $\bS^{d-1}$).

Let $Y_{l,i}$, $1\leq i \leq N_{d-1}(l)$, be an orthonormal basis of
$\cH_{l}(\bS^{d-1})$, i.e.
\[
\int _{\bS^{d-1}}Y_{l,i}(\xi) \overline{Y_{l,i'}(\xi)} d\sigma(\xi) = \delta_{i,i'}.
\]
Then the natural extensions of $Y_{l,i}$ on $B^d$ are defined by
$Y_{l,i}(x)=|x|^l Y_{l,i}\big( \frac{x}{|x|}\big)$ and satisfy
\begin{align*}
\int _{B^{d}}Y_{l,i}(x) \overline{Y_{l,i'}(x)} dx
&= \int_0^{1} r^{d-1} \int _{\bS^{d-1}}Y_{l,i}(r\xi) \overline{Y_{l,i'}(r\xi)} d\sigma(\xi)dr \\
&= \int_0^{1} r^{d+2l-1} \int _{\bS^{d-1}}Y_{l,i}(\xi) \overline{Y_{l,i'}(\xi)} d\sigma(\xi)dr
= \delta_{i,i'} \frac 1{2l+d}.
\end{align*}
For more details we refer the reader to  \cite{DUXU}.

The spherical harmonics on $\bS^{d-1}$ and orthogonal polynomials on $B^d$ are naturally related
to Gegenbauer polynomials.
The kernel of the orthogonal projector onto $\cH_n(\bS^{d-1})$ can be written as (see e.g. \cite{STW}) if ${N_{d-1}(n)}$ is the dimension of $\cH_n(\bS^{d-1})$:
\begin{equation}\label{legen}
\sum_{i=1}^{N_{d-1}(n)} Y_{l, i}(\xi ) \overline{Y_{l, i}(\theta)}
= \frac{2n+ d-2}{(d-2)|\bS^{d-1}|} C^{\frac{d-2}{2}}_n (\langle \xi, \theta \rangle).
\end{equation}
The ``ridge" Gegenbauer polynomials $C_n^{d/2}(\langle x, \xi  \rangle)$ are orthogonal
to $\Pi_{n-1}(B^d)$ in $\bL^2(B^d)$ and the kernel $L_n(x,y)$ of the orthogonal projector onto $\cV_n(B^d)$
can be written in the form (see e.g. \cite{petrush, xu7})
\begin{align}\label{orth-projector}
L_n(x,y)
&= \frac{2n+d}{|\bS^{d-1}|^2}
\int _{\bS^{d-1}}C^{d/2}_n ( \langle x, \xi  \rangle )C^{d/2}_n ( \langle y, \xi  \rangle )d\sigma(\xi)\\
&= \frac{(n+1)_{d-1}}{2^d \pi^{d-1}}
\int _{\bS^{d-1}}\frac{C^{d/2}_n ( \langle x, \xi  \rangle )
C^{d/2}_n ( \langle y, \xi  \rangle ) }{\|  C^{d/2}_n\|^2}d\sigma(\xi).\notag
\end{align}

The following important identities are valid for  ``ridge" Gegenbauer polynomials:
\begin{equation}\label{ridge-Gegen1}
\int_{B^d} C^{d/2}_n (\langle \xi, x \rangle)C^{d/2}_n (\langle \eta, x \rangle)dx
=\frac{h_n^{(d/2)}}{C_n^{d/2}(1)}C^{d/2}_n (\langle \xi, \eta \rangle),
\quad \xi, \eta\in \bS^{d-1},
\end{equation}
and, for $x\in B^d$, $\eta\in \bS^{d-1}$,
\begin{equation}\label{ridge-Gegen2}
\int_{\bS^{d-1}} C^{d/2}_n (\langle \xi, x \rangle)C^{d/2}_n (\langle \xi, \eta \rangle)d\sigma(\xi)
=|\bS^{d-1}| C^{d/2}_n (\langle \eta, x \rangle),
\end{equation}
see e.g. \cite{petrush}.
By (\ref{orth-projector}) and (\ref{ridge-Gegen2})
\[
L_n(x,\xi)
= \frac{(2n+d)}{|\bS^{d-1}|}  C^{d/2}_n ( \langle x, \xi  \rangle ),
\quad \xi\in \bS^{d-1},
\]
and again by (\ref{orth-projector})
\[
\int _{\bS^{d-1}}L_n (x,\xi) L_n(y, \xi) d\sigma(\xi) = (2n+d)L_n(x,y).
\]

\subsubsection{The SVD of the Radon transform}\label{SVD}

Assume that
$\{Y_{l,i}: 1\leq i \leq N_{d-1}(l)\}$ is an orthonormal basis for $\cH_{l}(\bS^{d-1})$.
Then it is standard and easy to see that the family of polynomials
$$
f_{k,l,i} (x)
= (2k+d)^{1/2}P_j^{(0,\, l + d/2 -1)} (2|x|^2-1)Y_{l, i}(x),
\;\; 0\leq l \leq k,\; k-l =2j, \; 1\leq i \leq N_{d-1}(l),
$$
form an orthonormal basis of $\cV_k(B^d)$, see e.g. \cite{DUXU}.
On the other hand the collection
$$
g_{k,l,i}(\theta, s)
=  [h_k^{(d/2)}]^{-1/2}(1-s^2)^{(d-1)/2} C^{d/2}_k(s) Y_{l,i }(\theta),
\quad k\ge 0, \; l \ge 0, \; 1\leq i \leq N_{d-1}(l),
$$
is obviously an orthonormal basis of $\bL^2(\bS^{d-1}\times [-1,1], d\mu(\theta,s))$.

\noindent
Figure~\ref{fig:RadonSVD} displays a few $f_{k,l,i}$ and illustrates
their lack of localization.

\begin{figure}
  \centering
  \begin{tabular}{ccc}
$f_{3,1,0}$&&$f_{4,4,1}$\\
  \includegraphics[width=4.5cm]{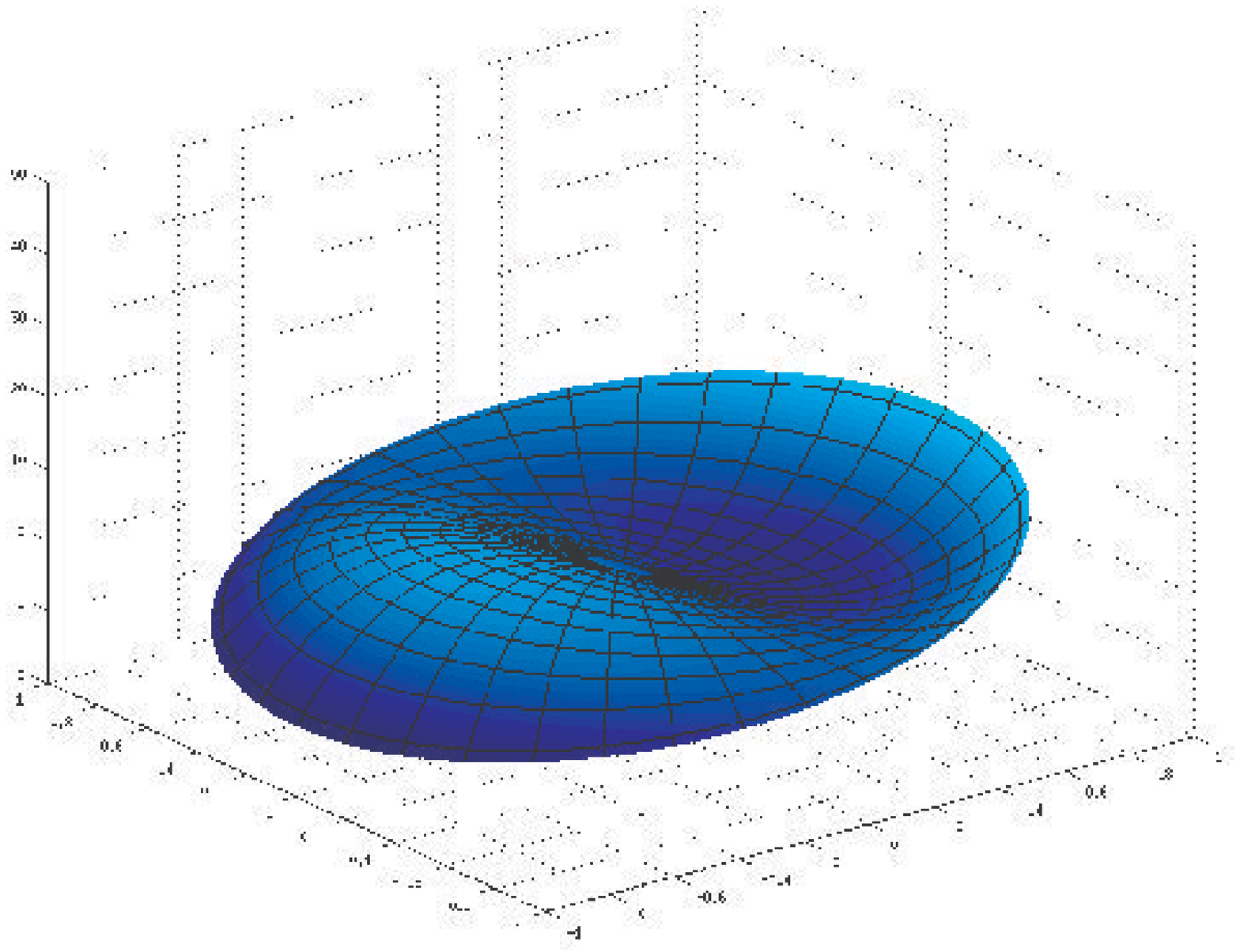} &\hspace*{.25cm}&
  \includegraphics[width=4.5cm]{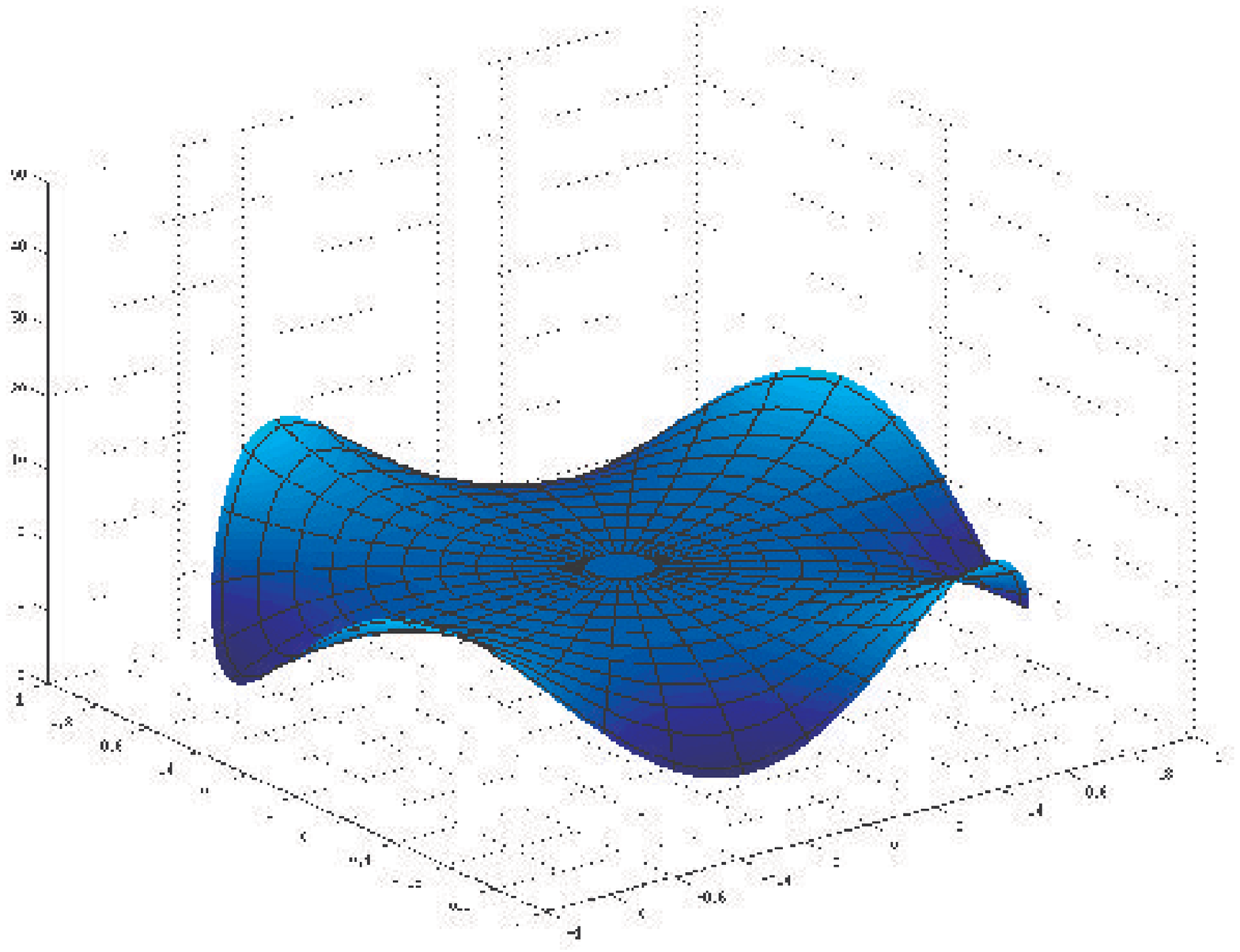}\\
  \includegraphics[width=4.5cm]{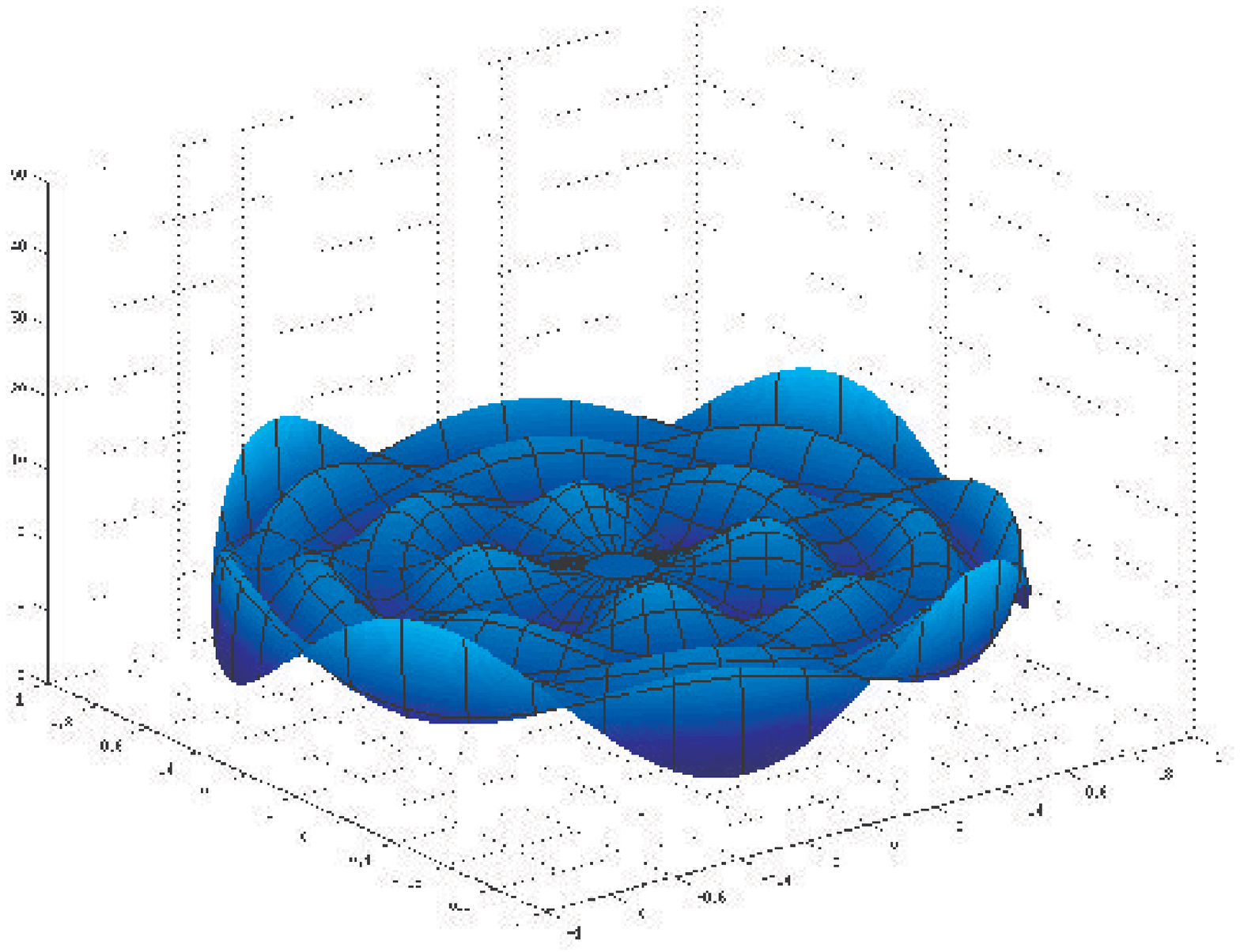} &\hspace*{.25cm}&
  \includegraphics[width=4.5cm]{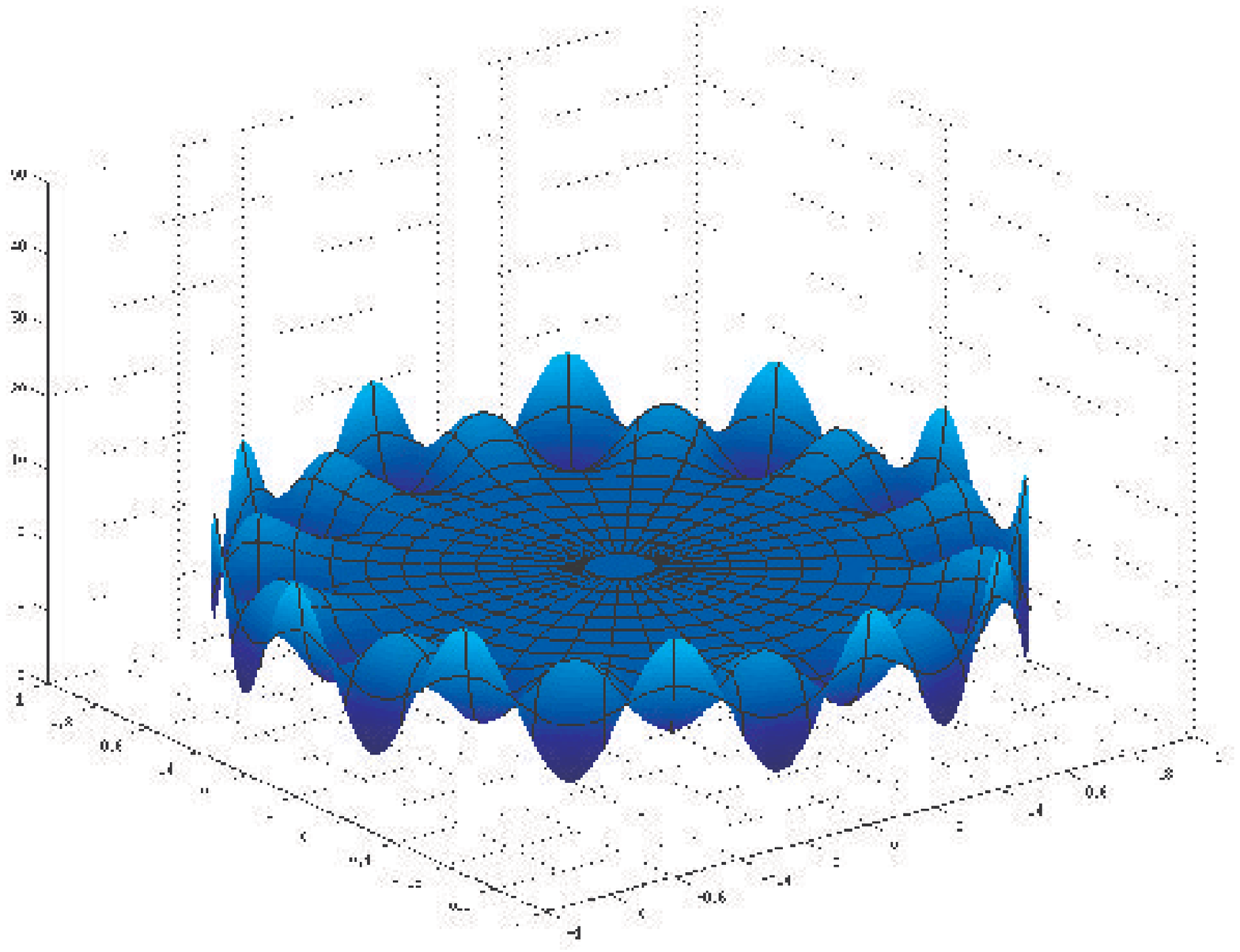}\\
$f_{14,4,1}$&&$f_{14,12,0}$
  \end{tabular}
  \caption{A few radon SVD basis elements (Low quality figure due to
    arXiv constraint)}
  \label{fig:RadonSVD}
\end{figure}

\smallskip
\noindent
The following theorem gives the SVD decomposition of the Radon transform.


\begin{theorem} \label{thm:SVD}
For any $f\in \bL^2(B^d)$
\begin{equation}\label{Radon1}
Rf=
\sum_{k\ge 0}\, \lambda_k
\sum_{0 \leq l \leq k, \,k-l \equiv 0 \,(\mod 2)}
\sum_{ 1  \leq i \leq N_{d-1}(l) }
\langle f, f_{k, l, i}\rangle g_{k,l,i}
\end{equation}
and for any $g\in \bL^2(\bS^{d-1}\times [-1,1], d\mu(\theta,s))$
\begin{equation}\label{Radon2}
R^*g=
\sum_{k\ge 0}\, \lambda_k
\sum_{0 \leq l \leq k, \,k-l \equiv 0 \,(\mod 2)}
\sum_{ 1  \leq i \leq N_{d-1}(l) }
\langle g, g_{k, l, i}\rangle_\mu f_{k,l,i}.
\end{equation}
Furthermore, for $f\in \bL^2(B^d)$
\begin{equation}\label{Radon3}
f= \sum_{k\ge 0}\, \lambda_k^{-1}
\sum_{0 \leq l \leq k, \,k-l \equiv 0 \,(\mod 2)}
\sum_{ 1  \leq i \leq N_{d-1}(l) }
\langle Rf, g_{k, l, i}\rangle_\mu f_{k,l,i}.
\end{equation}
In the above identities the convergence is in $\bL^2$ and
\begin{equation}\label{lambda-k}
\lambda_k^2 = \frac{2^d\pi^{d-1}}{(k+1)(k+2)\dots(k+d-1)}
= \frac{2^d\pi^{d-1}}{(k+1)_{d-1}}\sim k^{-d+1}.
\end{equation}

\end{theorem}


\begin{rem}
Observe that
if $ k\ge 0$, $0 \leq   l \leq k$, $k-l \not\equiv 0 \,(\mod 2)$ , and $1  \leq i \leq N_{d-1}(l)$,
then
$
R^*f_{k,l,i} = 0.
$
\end{rem}

For the proof of this result, we refer the reader to \cite{NATT} and \cite{xu7}.
We will only show in the appendix that $\lambda_k^2$ has the value given in (\ref{lambda-k}).

\section{Linear estimators built upon the SVD}\label{estimators}
\setcounter{equation}{0}

\subsection{The general idea}

In a general noisy inverse model
 \[
dY_t = K f dt + \eps dW_t,
\]
where $K$ is a linear operator mapping  $f\in \bH\mapsto Kf\in\bK$,
and $\bH$ and $\bK$ are two Hilbert spaces,  the SVD
yields a family of linear estimators via the following
classical scheme.

Suppose $K$ has an SVD
$$
Kf=\sum_m \sigma_m \langle f, e_m \rangle e_m^*,
\quad f\in \bH,
$$
where $\{e_m\}$ and  $\{e^*_m\}$ are orthonormal bases for $\bH$ and $\bK$, respectively,
and
$Ke_m=\sigma_m e^*_m$ and $K^*e^*_m =\sigma_m e_m$ with $K^*$ being the adjoint operator of $K$.
We also assume that $\sigma_m \to 0$.
Then, if $\sigma_m\not=0$,
\begin{align*}\int e^*_mdY_t
&= \int  K f \cdot e^*_mdt + \eps \int e^*_m dW_t
= \int   f \cdot K^* e^*_mdt + \eps \int e^*_mdW_t\\
&=\sigma_{m} \int   f    e_mdt + \eps \int e^*_mdW_t
\end{align*}
and hence
\begin{equation}\label{fem}
\frac 1{\sigma_{m}} \int e^*_mdY_t  =  \langle f, e_m\rangle +
\frac{\eps}{\sigma_{m}}  \int e^*_m dW_t.
\end{equation}
In going further, suppose that $\{\phi_l\}$ is a tight frame for $\bH$.
Therefore, for any $f\in \bH$
\[
f=\sum_l \alpha_l\phi_l,
\quad
\alpha_l=\langle f, \phi_l\rangle.
\]
We can represent $\phi_{l}$ in the basis $\{e_m\}$:
\[
 \phi_{l}  = \sum_{m}\langle \phi_{l}, e_m \rangle  e_m=
  \sum_{m} \gamma^{m}_{l}  e_m
\]
and hence
\[
\alpha_{l} = \sum_{m} \gamma^{m}_{l} \langle f, e_m\rangle.
\]
On account of (\ref{fem}) this leads to the estimator
\begin{equation}\label{flin}
\hat{f}_N= \sum_{l\le N}\hat{\alpha}_{l} \phi_{l}
\quad\mbox{with}\quad
\hat{\alpha}_{l} = \sum_{m} \gamma^{m}_{l}\frac 1{\sigma_{m}} \int e^*_m dY_t,
\end{equation}
where $N$ is a parameter.
By (\ref{fem}) we have
\begin{align*}
\hat{\alpha}_{l}
= \sum_{m} \gamma^{m}_{l}\langle f, e_m \rangle
+ \sum_{m} \gamma^{m}_{l}\frac{\eps}{\sigma_{m}}  \int e^*_m dW_t
= \alpha_{l} + Z_{l},
\end{align*}
where $Z_{l}$ has a normal distribution $N(0, \sum_{m} (\gamma^{m}_{l})^2\frac{\eps^2}{\sigma^2_{m}})$.
In this scheme the factors $\frac 1{\sigma_m^2}$,
which are inherent to the inverse model, bring instability by inflating the variance.

The selection of the frame $\{\phi_l\}$ is critical for the method described above.
The standard SVD method corresponds to the choice $\phi_l=e_l$.
This SVD method is very attractive theoretically and can be shown to be asymptotically optimal
in many situations
(see Dicken and Maass \cite{dicma}, Math\'{e} and Pereverzev \cite{MR1984890}
together with their nonlinear counterparts
Cavalier and Tsybakov \cite{cavalier_tsybakov}, Cavalier et al \cite{cgpt},
Tsybakov \cite{MR1769957}, Goldenschluger and Pereverzev \cite{MR2047686},
Efromovich and Kolchinskii \cite{MR1872847}).
It also has the  big advantage of performing a quick and stable inversion of the operator $K$.
However, while the SVD bases are fully adapted to describe the operator $K$,
they are usually not quite appropriate for accurate description of the solution of the problem
with a small number of parameters.
Although the SVD method is suitable for estimating the unknown function $f$ with an $L^2$-loss,
it is also rather inappropriate for other losses.
It is also restricted to functions which are well represented in terms of the Sobolev space
associated to the  SVD basis.
Switching to an arbitrary frame $\{\phi_m\}$, however,
may yield additional instability through the factors $(\gamma^{m}_{l})^2$'s.

Our idea is to utilize a frame $\{\phi_l\}$ which
is compatible with the SVD basis $\{e_m\}$, allowing to keep the variance within reasonable bounds,
and has elements with superb space localization and smoothness, guaranteeing excellent
approximation of the unknown function $f$.
In the following we implement the above described method to the inversion of the Radon transform.
We shall build upon the frames constructed in \cite{pxuball} and called ``needlets".

\subsection{Construction of needlets on the ball}

In this part we construct the building blocks of our estimator.
We will essentially follow the construction from \cite{pxuball}.

\subsubsection{The orthogonal projector \boldmath $L_k$ on $\cV_k(B^d)$ .}

Let $\{f_{k,l,i}\}$ be
the orthonormal basis of $\cV_k(B^d)$ defined in \S\ref{SVD}.
Denote by  $T_k$ the index set of this basis, i.e.
$
T_k=\{(l, i): 0\leq l \leq k,\, l\equiv k (\mod 2), \, 0\leq i \leq N_{d-1}(l)\}$.
Also, set  $\nu=d/2-1$. Then the orthogonal projector of $\bL^2(B^d)$ onto $\cV_k(B^d)$
can be written in the form
$$
L_kf = \int_{B^d} f(y) L_k(x,y) dy
\quad\mbox{with}\quad
L_k(x,y) = \sum_{l,i \in T_k}  f_{k,l,i}(x) f_{k,l,i}(y).
$$
Using (\ref{legen}) $L_k(x,y)$ can be written in the form
\begin{align}\label{EqF}
&L_k(x,y)\\ 
& =(2k+d)  \sum_{l\le k,\, k-l\equiv 0(2)} P_j^{(0, l +\nu)} (2|x|^2-1)  |x|^l
P_j^{(0, l +\nu)} (2|y|^2-1)|y|^l
\sum_{i} Y_{l, i}\Big(\frac x{|x|}\Big)Y_{l,i}\Big(\frac y{|y|}\Big) \nonumber\\
&= \frac{ (2k+d)}{|\bS^{d-1}|} \sum_{l\le k,\, k-l\equiv 0(2)}
P_j^{(0, l +\nu)} (2|x|^2-1)|x|^l P_j^{(0, l +\nu)} (2|y|^2-1) |y|^l \Big(1+\frac l\nu\Big)
C^\nu_l \Big(\Big\langle \frac x{|x|},  \frac y{|y|}  \Big\rangle \Big). \nonumber
\end{align}
Another representation of $L_k(x,y)$ has already be given in (\ref{orth-projector}).
Clearly
\begin{equation}\label{pro}
\int_{B^d} L_m(x,z) L_k(z,y) dz = \delta_{m,k} L_m(x,y)
\end{equation}
and for $f\in \bL^2(B^d)$
\begin{equation}\label{orthog-dec}
f=\sum_{k\ge 0} L_kf
\quad\mbox{and}\quad
\|f\|_2^2 =\sum_{k} \|L_k f\|_2^2 =\sum_{k} \langle L_k f, f \rangle.
\end{equation}

\subsubsection{Smoothing}\label{defAB}

Let $a\in C^\infty[0, \infty)$ be a cut-off function such that
$0 \leq a \leq 1$,
$a(t)=1$ for $t\in [0, 1/2]$ and $\supp a \subset [0, 1]$.
We next use this function to introduce a sequence of operators on $\bL^2(B^d)$.
For $j\ge 0$ write
\[
 A_j f (x) = \sum_{k\ge 0} a\Big(\frac k{2^j}\Big)  L_k f(x) = \int_{B^d} A_j(x,y)f(y) dy
\quad \text{ with } \; A_j(x,y)=  \sum_{k} a\Big(\frac k{2^j}\Big)  L_k (x,y).
\]
Also, we define $B_jf = A_{j+1}f-A_jf$.
Then setting $b(t) = a(t/2) - a(t)$ we have
\[
 B_j f (x) = \sum_{k} b\Big(\frac k{2^j}\Big)L_k f(x) = \int_{B^d} B_j(x,y)f(y) dy
\quad \text{ with } \;
B_j(x,y)=  \sum_{k} b\Big(\frac k{2^j}\Big)L_k (x,y).
\]
Evidently, for $f\in \bL^2(B^d)$
\begin{equation}{\label{pro1}}
\langle A_jf,f \rangle = \sum_{k} a\Big(\frac k{2^j}\Big) \langle  L_k f, f \rangle \leq   \| f\|_2^2
\end{equation}
and
\begin{equation}{\label{pro2}}
\lim_{j\rightarrow \infty } \| A_jf-f\|_2
=\lim_{j\rightarrow \infty } \| ( A_0 +\sum_{m=0}^{j-1} B_m)f-f\|_2
=0.
\end{equation}

An important result from \cite{pxuball} (see also \cite{pxukball}) asserts that
the kernels $A_j(x, y)$, $B_j(x, y)$ have nearly exponential localization, namely,
for any $M>0$ there exists a constant $c_M>0$ such that
\begin{equation}\label{EqFond}
|A_j(x,y)|, |B_j(x,y)| \leq C_M \frac{ 2^{jd}}{(1+2^jd(x,y))^M \sqrt{ W_j(x)} \sqrt{ W_j(y)} },
\quad x, y\in B^d,
\end{equation}
where
$
W_j(x) = 2^{-j} +\sqrt{ 1-|x|^2}, ~~|x|^2 =|x|^2_{d} =\sum_{i=1}^d x_i^2,
$
and
\begin{equation}\label{def-distance}
d(x,y)=  \Arccos( \langle x,y\rangle + \sqrt{ 1-|x|^2}\sqrt{ 1-|y|^2}),
\quad
\langle x,y\rangle = \sum_{i=1}^d x_iy_i.
\end{equation}

 The left part (1) of Figure~\ref{fig:NeedletCompare} illustrates this concentration: it displays the influence of a point $x$ to the
 value of $B_jf$ at a second point $y_0$, namely the values of
 $B_j(x,y_0)$ for a fixed $y_0$ and $j=4$. This influence peaks at $y_0$ and
 vanishes exponentially fast to $0$ as soon as one goes away from
 $y_0$.
The central part (2) of Figure~\ref{fig:NeedletCompare} shows the
modification of the concentration when $j$ is set to a large value ($j=6$).
The right part (3) of Figure~\ref{fig:NeedletCompare} shows
 the lack of  concentration of $B_j$ when the cut-off function $a$
 used is far from being $C^\infty$. The resulting kernel still peaks
 at $y_0$ but the value of $B_jf$ at $y_0$ is strongly influenced by
 values far away from $y_0$.

 \begin{figure}
   \centering

\begin{tabular}{cccccc}
$B_j(x,y_0)$
&\parbox[c]{4cm}{\includegraphics[width=4cm]{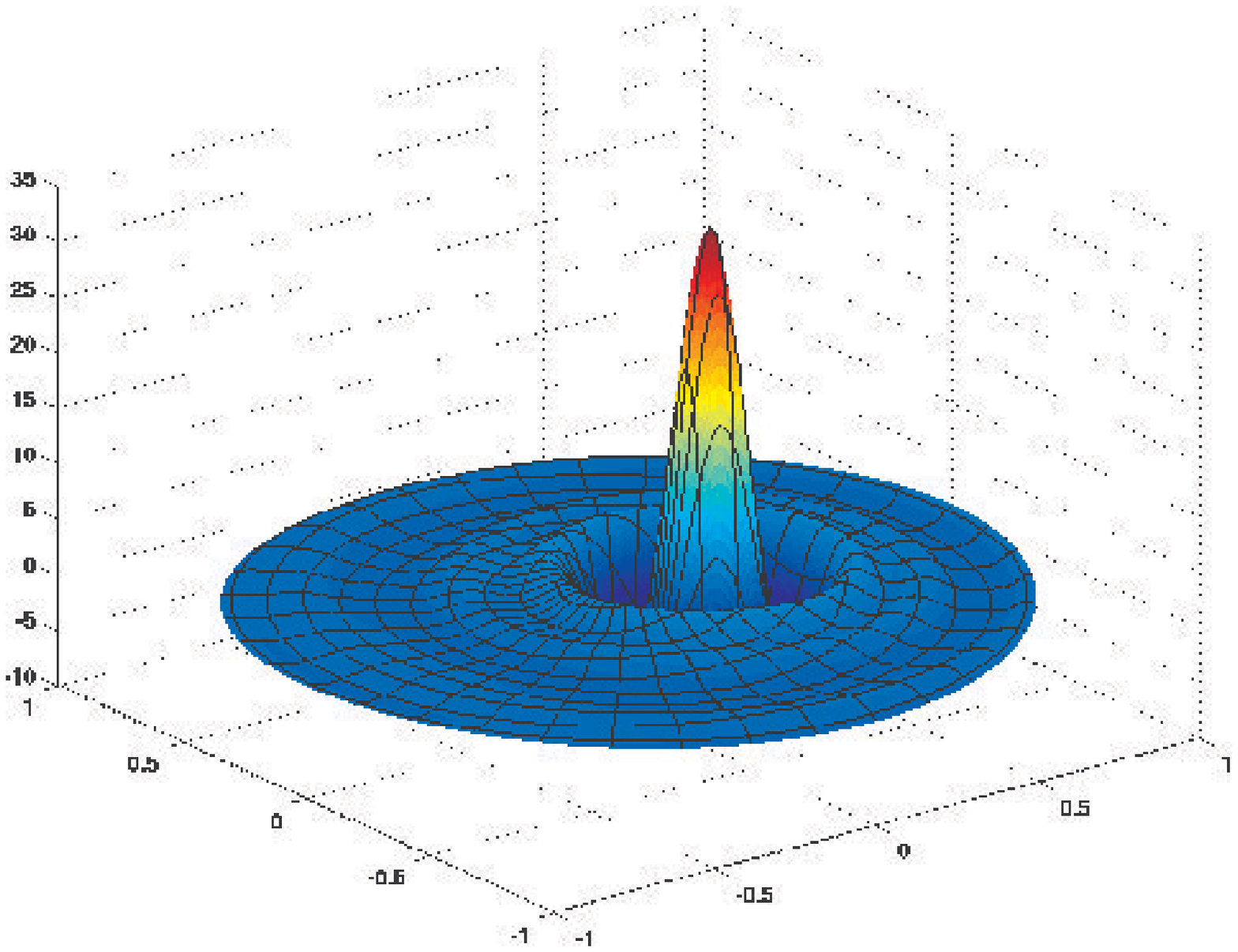}}
 && \parbox[c]{4cm}{\includegraphics[width=4cm]{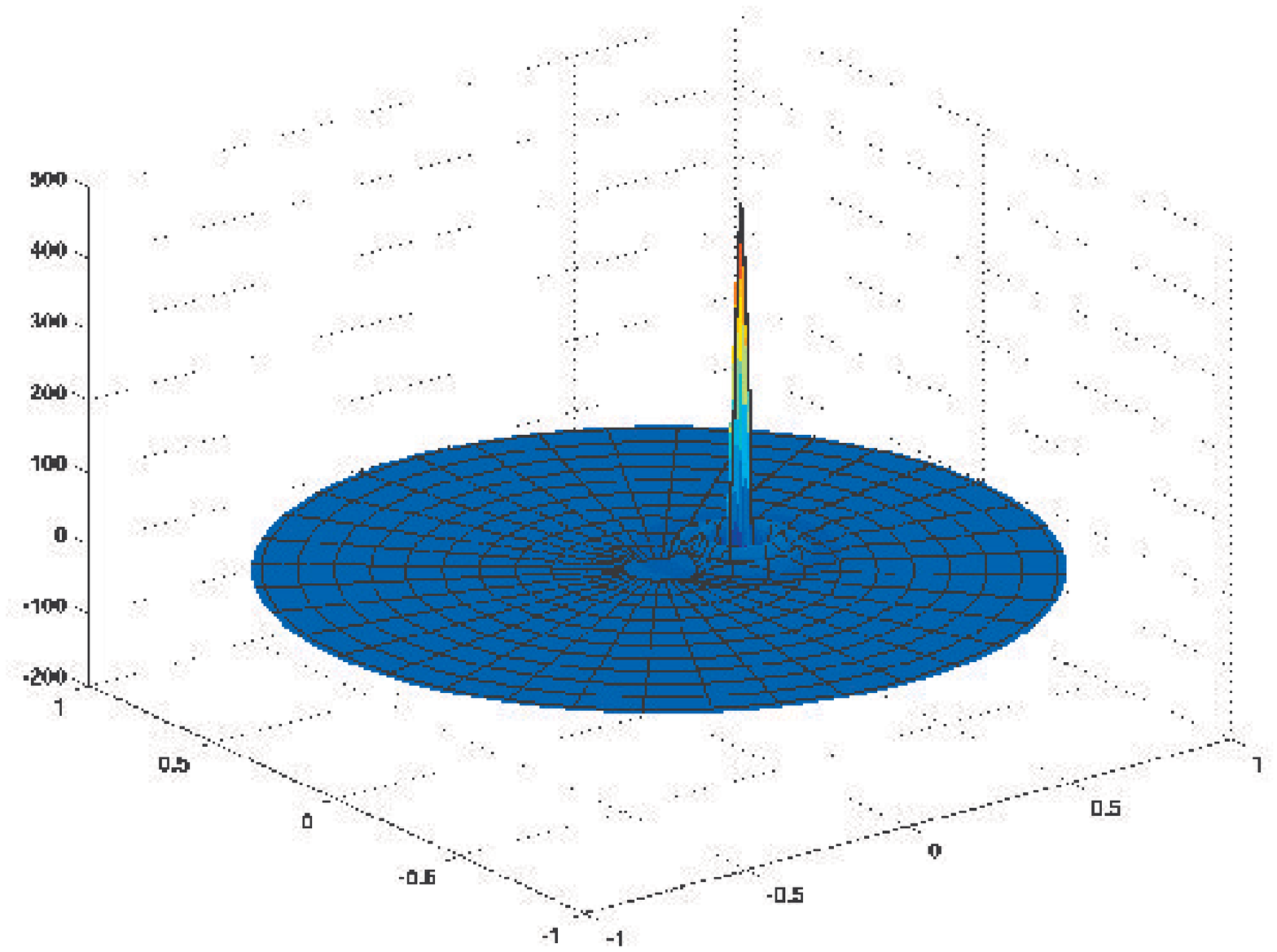}}
 && \parbox[c]{4cm}{\includegraphics[width=4cm]{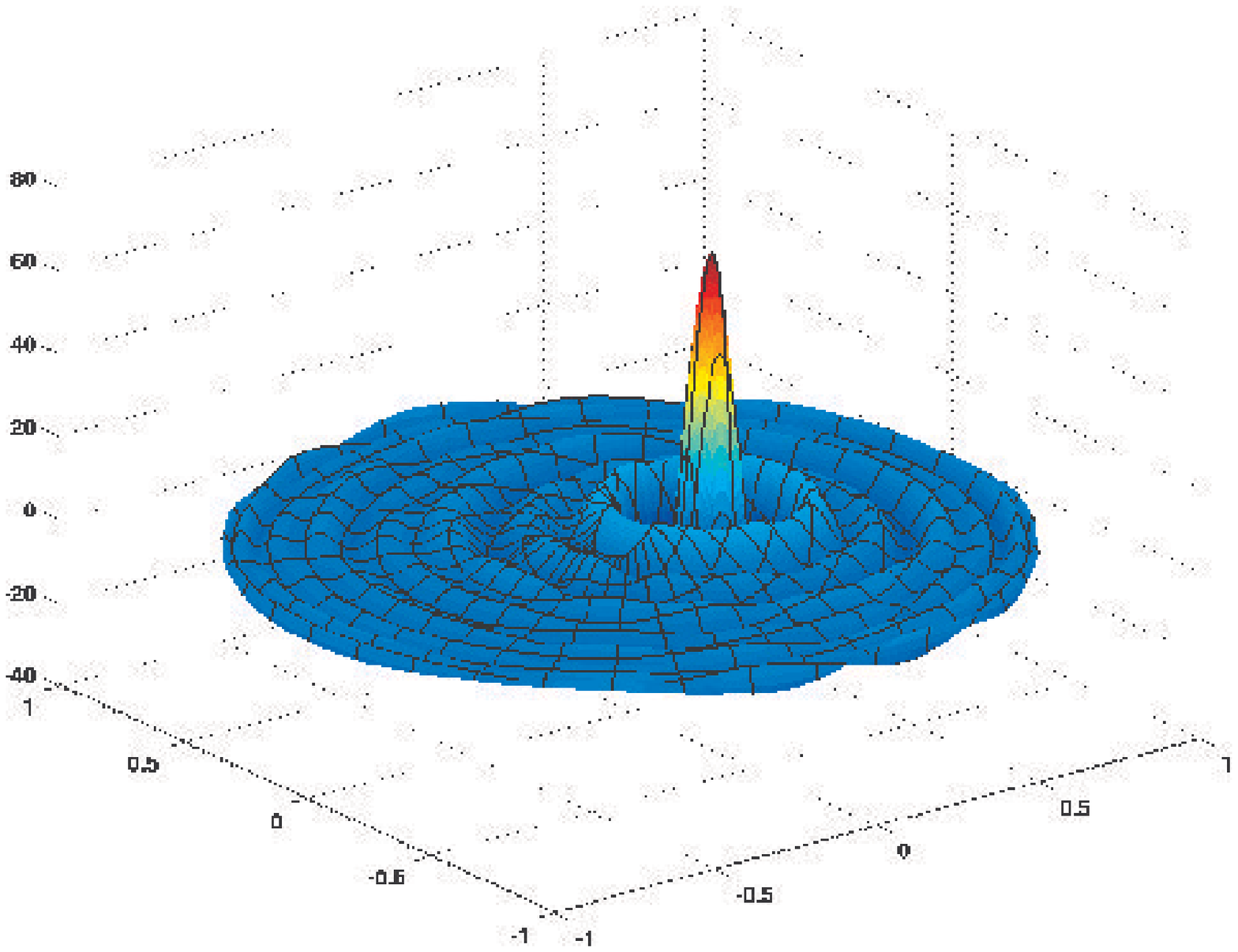}}\\
j & 4 && 6 && 4\\
$a$ & $C^\infty$&& $C^\infty$ && non smooth \\
& (1) && (2) && (3)
\end{tabular}

   \caption{Smoothing kernel $B_j(x,y_0)$ for a fixed $y_0$ for (1) a
     $C^\infty$ cut-off function $a$ with $j=4$
     , (2) for the same $a$ with $j=6$ and (3) for
 a 
     non smooth cut-off function $a$ with $j=4$ (Low quality figure due to
    arXiv constraint)
}
   \label{fig:NeedletCompare}
 \end{figure}


\begin{rem}\label{boulesphere}
At this point it is important to notice the following correspondence which will be
used in the sequel.
For $\bS_+^d = \{(x,z) \in \bR^d \times \bR^+ , ~|x|^2_d + z^2 =1\}$, we have  the natural bijection
 \[
 x\in B^d \mapsto \tilde{x} =(x, \sqrt{1-|x|^2})
\quad\mbox{ and }\quad
d(x,y) = d_{\bS_+^d }(\tilde{x}, \tilde{y}),
\]
where  $ d_{\bS_+^d }$ is the geodesic distance on  $\bS_+^d. $
 \end{rem}

\subsubsection{Approximation}\label{approx}

Here we discuss the approximation properties of the operators $\{A_j\}$.
We will show that in a sense they are operators of ``near best" polynomial $L^p$-approximation.
Denote by $E_n(f,p)$ the best $L^p$-approximation of $f\in \bL^p(B^d)$ from $\Pi_n$,
i.e.
\begin{equation}\label{def-Enp}
E_n(f,p) = \inf_{P \in \Pi_n} \|f-P\|_p.
\end{equation}
Estimate (\ref{EqFond}) yields (cf. (\cite[Proposition 4.5]{pxuball})
$$
\int_{B^d}|A_j(x, y)|dy \le c^*,
\quad x\in B^d, \; j\ge 0,
$$
where $c^*$ is a constant depending only on $d$.
Therefore, the operators $A_j$ are (uniformly) bounded on $\bL^1(B^d)$ and $\bL^\infty(B^d)$,
and hence, by interpolation, on $\bL^p(B^d)$, $1\le p\le \infty$, i.e.
\begin{equation}\label{bound-oper}
\|A_jf\|_p\le c^*\|f\|_p, \quad f\in \bL^p(B^d).
\end{equation}
On the other hand, since $a(t)=1$ on $[0, 1/2]$ we have
$A_jP=P$ for $P\in \Pi_{2^{j-1}}$.
We use this and (\ref{bound-oper}) to obtain, for $f\in \bL^p(B^d)$
and an arbitrary polynomial $P\in \Pi_{2^{j-1}}$,
$$
\| f-A_j f\|_p
= \|f-  P + P-A_jf\|_p
\le \|f- P\|_p + \|A_j(P-f)\|_p
\le (1+ c^*)\|f- P \|_p = K\|f- P \|_p.
$$
Consequently,
$\|f-A_jf\|_p \le  K E_{2^{j-1}}(f,p)$.
In the opposite direction,
evidently, $A_j f \in \Pi_{2^j}$ and hence
$
E_{2^j}(f,p) \leq  \|f-A_j f\|_p.
$
Therefore, for $f\in \bL^p(B^d)$, $1\le p\le\infty$,
\begin{equation}\label{Apoly}
E_{2^j}(f,p) \leq  \|f-A_jf\|_p \le  K E_{2^{j-1}}(f,p).
\end{equation}
These estimates do not tell the whole truth about the approximation power of $A_j$.
It is rather obvious that because of the superb localization of the kernel $A_j(x, y)$
the operator $A_j$ provides far better rates of approximation than $E_{2^{j-1}}(f,\infty)$
away from the singularities of $f$.

In contrast, the kernel $S_j(x, y)= \sum_{0\le k \le 2^j}L_k(x, y)$ of
the orthogonal projector $S_j$ onto $\Pi_{2^j}$
is poorly localized and hence $S_j$ is useless for approximation in $L^p$, $p\ne 2$.
This partially explains the fact that the traditional SVD estimators perform
poorly in $L^p$-norms when $p\ne 2$.

\subsubsection{Splitting  procedure}
 Let us define
\[
C_j(x,y) =\sum_{m} \sqrt{a\Big(\frac m{2^j}\Big)}  L_m (x,z) )
\quad\mbox{and}\quad
D_j(x,y)=\sum_{m} \sqrt{b\Big(\frac m{2^j}\Big)}  L_m (x,z).
\]
Note that $C_j$ and $D_j$ have the same localization as the localization of $A_j$, $B_j$
in (\ref{EqFond}) (cf. \cite{pxuball}).
Using (\ref{pro}), we get the desired splitting
  \begin{equation}\label{D1}
  A_j(x,y) = \int_{B^d}  C_j(x,z) C_j(z,y) dz
\end{equation}
   and
  \begin{equation}\label{D2}
   B_j(x,y) = \int_{B^d} D_j(x,z) D_j(z,y) dz.
  \end{equation}
Obviously
$z \mapsto  C_j(x,z) C_j(z,y)$
is a polynomial of degree $< 2^{j+1}$ and
$z \mapsto D_j(x,z) D_j(z,y)$ is a polynomial of degree $< 2^{j+2}$.
The next step is to discretize the kernels $A_j(x, y)$ and $B_j(x, y)$.

\subsubsection{Cubature formula and discretization}
\label{ballsphere}

To construct the needlets on $B^d$ we need one more ingredient -
a cubature formula on $B^d$ exact for polynomials of a given degree.

Recall first the bijection between the ball
$B^d $ (equipped with the usual Lebesgue measure)
and the unit  upper hemisphere in $\bR^{d+1}$:
\[
\bS^d_+ = \{  (x,y), ~x\in \bR^d, ~0\leq y\leq 1, ~~|x|_d^2 +y^2 =1 \}
\]
equipped with $d\sigma$ the usual surface measure.

\[
T:   (x,y) \in \bS^d_+  \mapsto   x\in \bR^d
\]
and
\[
T^{-1} :  x \in \bR^d   \mapsto \tilde{x}=  (x, \sqrt{1-|x|_d^2}) \in
\bS^d_+
\]
Applying the substitution $T$ one has (see e.g.  \cite{cubxu})
\begin{equation}
\int_{\bS^d_+ } F(x,y) d\sigma(x,y) =  \int_{B^d} F(x, \sqrt{1-|x|^2})  \frac{dx}{\sqrt{1-|x|^2}}
\end{equation}
and hence for $ f: \bR^d  \mapsto  \bR$
\begin{equation}
\int_{ (x,y) \in S^d_+ } f(x) y d\sigma_d(x,y)
=  \int_{x \in B^d} f(x)  \sqrt{1-|x|^2_d} \frac{dx}{\sqrt{1-|x|^2_d}}
=  \int_{B^d} f(x) dx.
\end{equation}
Therefore, given a cubature formula on $\bS^d$ one can easily derive a cubature formula on $B^d$.
Indeed,
suppose we have a cubature formula on $\bS_+^d$ exact for all polynomials of degree $n+1$,
i.e, there exist $ \tilde{\chi}_n   \subset \bS^d_+ $ and
coefficients  $\lambdaquadra_{\tilde \xi} >0$, $ \tilde{\xi} \in  \tilde{\chi}_n$,  such that
\[
\int _{\bS_+^d}P(u) d\sigma(u) =
\sum_{\tilde{\xi}  \in  \tilde{\chi}_n} \lambdaquadra_{ \tilde \xi} P( \tilde \xi)
\quad
\forall P\in \Pi_{n+1}(\bR^{d+1}).
\]
If  $P\in \Pi_{n}(\bR^{d})$ then  $ P(x)y \in \Pi_{n+1}(\bR^{d+1}) $ and hence
\[
\sum_{\tilde \xi \in  \tilde{\chi}_n}   \lambdaquadra_{\tilde \xi} P(\xi)  \sqrt{ 1-\xi^2}
= \int _{\bS^d_+}P(x) y d\sigma
= \int_{B^d}   P(x)   dx.
\]
Thus the projection  $\chi_n$ of $ \tilde{\chi}_n$ on $B^d$  is the set of nodes and
the associated coefficients given by $\lambdaquadra_\xi= \sqrt{ 1-\xi^2}  \lambdaquadra_{\tilde \xi} $
induce a cubature formula on $B^d$ exact for $\Pi_n(\bR^d)$.

The following proposition follows from results in \cite{pxuball} and \cite{cubxu}.


\begin{proposition}\label{prop:CUB}
Let $\{B(\tilde\xi_i,\rho):\,i\in I\}$ be  a maximal family of
disjoint spherical caps of radius $\rho= \tau 2^{-j}$ with
centers on the hemisphere $\bS_+^d$.
Then for sufficiently small $0<\tau\le 1$ the set of points $\chi_j=\{\xi_i: i\in I\}$
obtained by projecting the set $\{\tilde\xi_i:\, i\in I\}$ on $B^d$
is a set of nodes of a cubature formula which is exact for  $\Pi_{2^{j+2}}(B^d)$.
Moreover, the coefficients $\lambdaquadra_{\xi_i}$ of this cubature formula are positive and
$\lambdaquadra_{\xi_i}\sim W_j(\xi_i)2^{-jd}$. Also, the cardinality $\#\chi_j\sim 2^{jd}$.
\end{proposition}

\subsubsection{Needlets}\label{needlets}

Going back to identities \eqref{D1} and \eqref{D2} and applying the cubature formula
described in Proposition~\ref{prop:CUB}, we get
\begin{align*}
A_j(x,y) &= \int_{B^d}  C_j(x,z) C_j(z,y) dz
=\sum_{\xi \in \chi_j} \lambdaquadra_{\xi} C_j (x,\xi)C_j ( y,\xi)
\quad\mbox{and}\\
B_j(x,y) &= \int_{B^d}  D_j(x,z) D_j(z,y) dz
=\sum_{\xi \in \chi_j} \lambdaquadra_{\xi} D_j (x,\xi)D_j ( y,\xi).
\end{align*}
We define the {\bf father needlets} $\phi_{j,\xi}$ and
the {\bf mother needlets} $\psi_{j,\xi}$ by
\[
 \phi_{j,\xi} (x) =  \sqrt{ \lambdaquadra_{\xi}} C_j (x,\xi)
 \quad\text{and}\quad
\psi_{j,\xi} (x) =  \sqrt{ \lambdaquadra_{\xi}} D_j (x,\xi),
\quad \xi\in\chi_j, \;j\ge 0.
\]
We also set
$\psi_{-1,0}=\frac{\ONE_{B^d}}{|B^d|}$ and $\chi_{-1}=\{0\}$.
From above it follows that
\[
A_j(x,y) =\sum_{\xi \in \chi_j}  \phi_{j,\xi} (x) \phi_{j,\xi} (y),
\quad
B_j(x,y) =\sum_{\xi \in \chi_j}  \psi_{j,\xi} (x) \psi_{j,\xi} (y).
\]
Therefore,
\begin{equation}\label{rep-Aj}
 A_j f (x) = \int_{B^d} A_j(x,y) f(y) dy=
 \sum_{\xi \in \chi_j} \langle f, \phi_{j,\xi} \rangle  \phi_{j,\xi}
 =\sum_{\xi \in \chi_j} \alpha_{j,\xi}   \phi_{j,\xi},
 \quad \alpha_{j,\xi}= \langle f, \phi_{j,\xi} \rangle.
\end{equation}
and
\begin{equation}\label{rep-Bj}
 B_j f (x) = \int_{B^d} B_j(x,y) f(y) dy=
 \sum_{\xi \in \chi_j} \langle f, \psi_{j,\xi} \rangle  \psi_{j,\xi}
 =\sum_{\xi \in \chi_j} \beta_{j,\xi}   \psi_{j,\xi},
\quad \beta_{j,\xi}=\langle f, \psi_{j,\xi} \rangle.
\end{equation}

By (\ref{rep-Aj}) and (\ref{pro1}) we have
 \[
\|\phi_{j,\xi} \|_2^2 \geq \langle A_j\phi_{j,\xi}  , \phi_{j,\xi}  \rangle
=\langle \sum_{\xi' \in \chi_j}
\langle  \phi_{j,\xi} ,  \phi_{j,\xi'} \rangle\phi_{j,\xi'}, \phi_{j,\xi}\rangle
=\sum_{\xi' \in \chi_j}| \langle  \phi_{j,\xi} ,
\phi_{j,\xi'}  \rangle |^2 \geq  \|\phi_{j,\xi} \|_2^4
\]
and hence
\begin{equation}\label{NL2}
 \|\phi_{j,\xi} \|_2 \leq 1.
\end{equation}
From (\ref{orthog-dec}) and the fact that
$\sum_{j\ge 0} b(t2^{-j})=1$ for $t\in [1, \infty)$, it readily follows that
\[
f = 
\sum_{j\ge -1} \sum_{\xi\in
\chi_j } \langle f,  \psi_{j,\xi} \rangle  \psi_{j,\xi},
\quad f\in\bL^2(B^d),
\]
and taking inner product with $f$ this leads to
\[
\|f\|_2^2 =  
\sum_j \sum_{\xi\in \chi_j } |\langle f,  \psi_{j,\xi} \rangle|^2,
\]
which in turn shows that the family $\{\psi_{j,\xi}\}$ is a tight frame for $\bL^2(B^d)$
and consequently
\begin{equation}\label{NL3}
\|\psi_{j,\xi} \|_2^2 \geq  \|\psi_{j,\xi} \|_2^4,
\quad\mbox{i.e.}\quad
\|\psi_{j,\xi} \|_2  \leq 1.
\end{equation}

Observe that using the properties of the cubature formula from Proposition~\ref{prop:CUB},
estimate
(\ref{EqFond}) leads to the localization estimate (cf. \cite{pxuball}):
\begin{equation}\label{needlet-local}
|\phi_{j,\xi}(x)|, | \psi_{j,\xi}(x)| \leq C_M \frac{ 2^{jd/2}}{  \sqrt{W_j(\xi)} (1+2^j d(x,\xi))^M}
\quad \forall M>0.
\end{equation}
Nontrivial lower bounds for the norms of the needlets are obtained in \cite{pxukball}.
More precisely, in \cite{pxukball} it is shown that for $0<p \le \infty$
\begin{equation}\label{norm-est}
\|\psi_{j,\xi}\|_p \sim \|\phi_{j,\xi}\|_p
\sim \Big(\frac{2^{jd}}{W_j(\xi)} \Big)^{1/2-1/p},
\quad \xi\in\chi_j.
\end{equation}

We next record some properties of needlets which will be needed later on.
For convenience we will denote in the following by
$h_{j, \xi}$ either $\phi_{j,\xi}$ or $\psi_{j,\xi}$.


\begin{theorem}\label{thm:stability}
Let $1\leq p \leq \infty$ and $j\ge 1$.
The following inequalities hold
\begin{align}
\sum_{\xi \in \chi_j}  \|   h_{j,\xi} \|^p_p  &\leq c2^{j(dp/2 +  (p/2  - 2)_+)}
\quad\mbox{if}\quad  p\neq 4, \label{B}\\
\sum_{\xi \in \chi_j} \|h_{j,\xi} \|^p_p  &\leq cj 2^{jdp/2}
\quad\mbox{if}\quad  p=4, \label{BI}
\end{align}
and for any collection of complex numbers $\{d_\xi\}_{\xi\in \chi_j}$
\begin{equation}\label{B2prime}
\| \sum_{\xi  \in \chi_j}d_\xi h_{j,\xi} \|_p
\leq c \Big(\sum_{\xi\in \chi_j} |d_\xi |^p \| h_{j,\xi} \|^p_p\Big)^{1/p}.
\end{equation}
Here $c>0$ is a constant depending only on $d$, $p$, and $\tau$.
\end{theorem}
To make our presentation more fluid we relegate the proof of this theorem to the appendix.

\subsection{Linear needlet estimator}\label{needlet-estim}

Our motivation for introducing the estimator described below
is the excellent approximation power of the operators $A_j$ defined in \S\ref{defAB}
and its compatibility with the Radon SVD.
We begin with the following representation of the unknown function $f$
$$
f = \sum_{k, l, i} \langle f, f_{k,l,i}\rangle\, f_{k,l,i},
$$
where the sum is over the index set
$\{(k,l,i): k\ge 0, 0\le l \le k, l\equiv k (\mod 2) ,1 \leq i \leq N_{d-1}(l)\}$.
Combining this with the definition of $A_j$ we get
$$
A_jf = \sum_{\xi \in  \chi_j} \langle  f ,  \phi_{j, \xi}(y)\rangle\, \phi_{j, \xi}
= \sum_{\xi \in  \chi_j} \alpha_{j, \xi}\phi_{j, \xi},
$$
where
$$
\alpha_{j, \xi} = \langle f, \phi_{j, \xi} \rangle
=\sum_{k, l, i} \gamma^{j,\xi}_{k,l,i} \langle f, f_{k,l,i}\rangle
=\sum_{k, l, i} \gamma^{j,\xi}_{k,l,i}\frac{1}{\lambda_k} \langle R(f), g_{k,l,i}\rangle_{\mu}
=\sum_{k, l, i} \gamma^{j,\xi}_{k,l,i}\frac{1}{\lambda_k} \int
g_{k,l,i} R(f) d\mu.
$$
Here
$\gamma^{j,\xi}_{k,l,i}=\langle f_{k,l,i} , \phi_{j,\xi}(y)\rangle$
can be precomputed.

It seems natural to us to define an estimator $\widehat{f}_j$ of the unknown function $f$ by
\begin{equation}\label{estimator-f}
\widehat{f}_j = \sum_{\xi \in \chi_j}\widehat{\alpha}_{j,\xi} \phi_{j,\xi},
\end{equation}
where
\begin{equation}\label{estimator-alpha}
\widehat{\alpha}_{j, \xi}
= \sum_{k, l, i}
\gamma^{j,\xi}_{k,l,i}\frac 1{\lambda_k }\int g_{k,l,i}dY.
\end{equation}
Here the summation is over
$\{(k,l,i): 0\le k < 2^j, 0\le l \le k, l\equiv k (\mod 2) ,1 \leq i \leq N_{d-1}(l)\}$
and $j$ is a parameter.

Some clarification is needed here.
The father and mother needlets, introduced in \S\ref{needlets},
are closely related but play different roles.
Both $\phi_{j, \xi}$ and $\psi_{j, \xi}$ have superb localization,
however, the mother needlets $\{\psi_{j, \xi}\}$ have multilevel structure
and, therefore, are an excellent tool for nonlinear n-term approximation of functions on the ball,
whereas the father needlets are perfectly well suited for linear approximation.
So, there should be no surprise that we use the father needlets for our linear estimator.

Furthermore, even if the needlets are central in the analysis of the estimator,
the estimator $\widehat{f}_j$ can be defined without them. Indeed,
\begin{align*}
\widehat{f}_j &= \sum_{\xi \in \chi_j} \sum_{k, l, i}
\gamma^{j,\xi}_{k,l,i}\frac 1{\lambda_k }\int g_{k,l,i}dY\,
\phi_{j,\xi}
\intertext{as all the sum are finite, their order can be interchanged,
yielding}
\widehat{f}_j  &= \sum_{k, l, i}
\frac 1{\lambda_k }\int g_{k,l,i}dY\,
\sum_{\xi \in \chi_j} \gamma^{j,\xi}_{k,l,i}
\phi_{j,\xi} = \sum_{k,l,i} \frac 1{\lambda_k }\int g_{k,l,i}dY\, A_j
f_{k,l,i}\\
\intertext{and thus the estimator is obtained by a simple
  componentwise multiplication on
 the SVD coefficients}
\widehat{f}_j  &=  \sum_{k,l,i} \frac {a\left(\frac{k}{2^j}\right)}{\lambda_k }\int g_{k,l,i}dY\,
f_{k,l,i}.
\end{align*}

However, as will be shown in the sequel, the precise choice of this smoothing allows to consider $\L_p$ losses and precisely  because of the localization properties of the atoms this approach will be extended using a thresholding procedure in a further work, using this time the mother wavelet.

\section{The risk of the needlet estimator}
\setcounter{equation}{0}

In this section we estimate the risk of the needlet estimator introduced above in
terms of the Besov smoothness of the unknown function.

\subsection{Besov spaces}

We introduce the Besov spaces of positive smoothness on the ball as  spaces
of $L^p$-approximation from algebraic polynomials.
As in \S\ref{approx} we will denote by $E_n(f,p)$ the best $L^p$-approximation of
$f\in \bL^p(B^d)$ from $\Pi_n$.
We will mainly use the notations from \cite{pxukball}.


\begin{definition}\label{defbesov}$\cite{pxukball}$
Let $0 <s<\infty$, $1\leq p \leq \infty$, and $0<q\leq \infty$.
The space $B^{s,0}_{p,q}$ on the ball is defined as the space of all functions
$f \in \bL^p(B^d)$ such that
$$
|f|_{B^{s,0}_{p,q}}=\Big( \sum_{n\geq 1} (n^s E_n(f,p))^q \frac 1n \Big)^{1/q} <\infty
\quad\mbox{if $q<\infty$,}
$$
and $|f|_{B^{s,0}_{p,q}} =\sup_{n\geq 1} n^s E_n(f,p) <\infty$ if $q=\infty$.
The norm on $B^{s,0}_{p,q}$ is defined by
\[
\|f\|_{B^{s,0}_{p,q}} =\|f\|_p + |f|_{B^{s,0}_{p,q}}.
\]
\end{definition}


\begin{rem}
From the monotonicity of $\{E_n(f,p)\}$ it readily follows that
$$
\|f\|_{B^{s,0}_{p,q}}  \sim \|f\|_p
+ \Big(\sum_{j\ge 0}  ( 2^{js} E_{2^j}(f,p))^q \Big)^{1/q}
$$
with the obvious modification when $q=\infty$.
\end{rem}

There are several different equivalent norms on the Besov space $B^{s,0}_{p,q}$.


\begin{theorem}\label{thm:Besov}
With indexes $s, p, q$ as in the above definition
the following norms are equivalent to the Besov norm $\|f\|_{B^{s,0}_{p,q}}$:
\begin{align*}
\mbox{$(i)$}\quad
\cN_1(f)&=\|f\|_p +\|  ( 2^{js} \|f-A_jf\|_p)_{j\ge 0}  \|_{l^q},\\
\mbox{$(ii)$}\quad
\cN_2(f)&=\|f\|_p +\|  ( 2^{js}  \|B_jf\|_p)_{j\ge 1}  \|_{l^q},\\
\mbox{$(iii)$}\quad
\cN_3(f)&=\|f\|_p +\|  ( 2^{js}\sum_{\xi \in \chi_j}
|\langle f, \psi_{j,\xi} \rangle |^p\| \psi_{j,\xi} \|^p_p)_{j\ge -1}  \|_{l^q}.
\end{align*}
\end{theorem}

\begin{proof}
The equivalence $\cN_1(f)\sim \|f\|_{B^{s,0}_{p,q}}$ is immediate from (\ref{Apoly}).

To prove that $\cN_2(f)\sim \cN_1(f)$, we recall that
$B_j = (A_{j+1}-A_j)$ (see \S\ref{defAB}) and hence
$
\|B_jf\|_p\le \|f- A_{j+1}f\|_p+\|f-A_jf\|_p
$
which readily implies
$
\cN_2(f)\le c\cN_1(f).
$
In the other direction, we have
$$
 \|f-A_jf\|_p = \| \sum_{l=j}^\infty B_l f\|_p  \leq
\sum_{l=j}^\infty \| B_lf\|_p.
$$
Assuming that $\cN_2(f)<\infty$ we have
$\| B_l (f)\|_p =\alpha_l 2^{-ls}$ with $\{\alpha_l\} \in l^q$.
Hence
$$
\sum_{l=j}^\infty \| B_l (f)\|_p =\sum_{l=j}^\infty \alpha_l 2^{-ls}
= 2^{-js}\sum_{l=j}^\infty \alpha_l 2^{-(l-j)s} =: 2^{-js} \beta_j
$$
and by the convolution inequality $\{\beta_j\} \in l^q$.
Therefore, $\cN_1(f)\le c\cN_2(f)$.

For the equivalence $\cN_3(f)\sim \|f\|_{B^{s,0}_{p,q}}$, see \cite[Theorem 5.4]{pxukball}.
\end{proof}

 \subsubsection{Comparison with the ``standard" Besov spaces}\label{compbesov}

 The classical Besov space $B^{s}_{p, q} (B^d)$  is defined through
 the $L^p$-norm of the finite differences:
 \[
\Delta_h f(x) =( f(x+h)- f(x)) \ONE_{x\in B^d}  \ONE_{x +h \in B^d}
\]
and in general
\[
\Delta^N_h f  (x)
=\ONE_{x\in B^d}  \ONE_{x + Nh \in B^d} \sum_{k=0}^N (-1)^{N-k} {N\choose k}f(x+kh).
\]
Then the $N$th modulus of smoothness in $L^p$ is defined by
\[
\omega^N_p(f,t) = \sup_{|h| \leq t} \| \Delta^N_h f  \|_p, \quad t>0.
\]
For $0 <s <N$, $1\le p\le \infty$, and $0<q \leq \infty$,
the classical Besov space $B^s_{p,q}$ is defined by the norm
\[
\|f \|_{B^s_{p,q}}
= \|f \|_p + \Big(\int_0^1  [t^s \omega^N_p(f,t)]^q \frac{dt}t\Big)^{1/q}
\sim \|f \|_p +\Big(\sum_{j=0}^\infty [2^{js}\omega^N_p(f,2^{-j})]^q\Big)^{1/q}
\]
with the usual modification for $q=\infty$.
It is well known that the definition of $B^s_{p,q}$ does not depend on $N$ as long as
$s<N$ \cite{HJKS}.
Moreover, the embedding
\begin{equation}
 B^{s}_{p, q} \subseteq B^{s,0}_{p, q},
\end{equation}
is immediate from the estimate
$
E_n(f,p)\leq c\omega^N_p(f, 1/n)
$
\cite{HJKS}.

\subsection{Upper bound for the risk of the needlet estimator}\label{upper-bound}


\begin{theorem}\label{thm:upper}
Let $1\leq p \leq \infty, 0<s<\infty$, and
assume that $f \in B^{s,0}_{p,\infty}$ with $\| f \|_{B^{s,0}_{p,\infty}} \leq M$.
Let
\[
\widehat{f}_J =  \sum_{\xi \in \chi_J}\widehat{\alpha}_{J,\xi} \phi_{j,\xi}
\]
be the needlet estimator introduced in \S\ref{needlet-estim},
where $J$ is selected depending on the parameters as described below.
\begin{enumerate}
\item
If $M 2^{-J(s+d)}\sim \eps$ when $p=\infty$, then
\[
\bE  \| f- \widehat{f_J} \|_\infty \leq  c_\infty  M^{^{\frac d{s+d}}}
 \eps^{\frac s{s+d}}  \sqrt{\log M/\eps}.
\]
\item
If $M2^{-Js} \sim \eps 2^{J(d-2/p)}$ when $4\leq p<\infty$, then
\[
\bE  \| f- \widehat{f_J} \|_p^p \leq c_p M^{ \frac{(d-2/p)p}{s+d-2/p}}
\eps^{\frac{sp}{s+d-2/p}},
\]
where when $p=4$ there is an additional factor $\ln (M/\eps)$ on the right.
\item
If $M2^{-Js}\sim \eps 2^{J(d-1/2)}$ when $1\leq p<4$, then
\[
\bE  \| f- \widehat{f_j} \|_p^p \leq c_p M^{\frac{(d-1/2)p}{s+d-1/2}}
\eps^{\frac{sp}{s+d-1/2}}.
\]
\end{enumerate}

\end{theorem}


\begin{rems}
\item
It will be shown in a forthcoming paper that the following rates of convergence are,
in fact, minimax, i.e. there exist positive constants $c_1$ and $c_2$ such that
\begin{align*}
\sup_{\| f \|_{B^{s,0}_{p,\infty}} \leq M}\inf_{\tilde f \, {\rm estimator}}\bE\| f- \tilde{f} \|_p^p
&\ge c_1\max \{\eps^{\frac{sp}{s+d-2/p}},\; \eps^{\frac{sp}{s+d-1/2}}\},\\
\sup_{\| f \|_{B^{s,0}_{\infty,\infty}} \leq M}\inf_{\tilde f \, {\rm estimator}} \bE \| f- \tilde{f} \|_\infty
&\ge c_2\eps^{\frac s{s+d}}\sqrt{\log 1/\eps}.
\end{align*}

\item
The case $p=2$ above corresponds to the standard SVD method which involves Sobolev spaces.
In this setting, minimax rates have already been established
(cf.  \cite{dicma}, \cite{MR1984890}   \cite{cavalier_tsybakov},  \cite{cgpt},
\cite{MR1769957},  \cite{MR2047686},  \cite{MR1872847});
these rates are $\eps^{\frac{2s}{s+d-1/2}}$.
Also, it has been shown that the SVD algorithms yield minimax rates.
These results extend (using straightforward comparisons of norms) to $L^p$ losses for $p<4$,
but still considering the Sobolev ball $\{\| f \|_{B^{s,0}_{2,\infty}} \leq M\}$
rather than the Besov ball $\{\| f \|_{B^{s,0}_{p,\infty}} \leq M\}$.
Therefore, our results can be viewed as an extension of the above results,
allowing a much wider variety of regularity spaces.

\item
The Besov spaces involved in our bounds are in a sense well adapted to our method.
However, the embedding results from Section \ref{compbesov} shows that the bounds from
Theorem~\ref{thm:upper} hold in terms of the standard Besov spaces as well.
This means that in using the Besov spaces described above,  our results are but stronger.

\item
In the case $p\ge 4$ we exhibit here new minimax rates of convergence,
related to the ill posedness coefficient of the inverse problem $\frac{d-1}2$
along with edge effects induced by the geometry of the ball.
These rates have to be compared with similar phenomena occurring in other inverse problems
involving Jacobi polynomials (e.g. Wicksell problem), see \cite{kppw}.

\end{rems}

\subsection{Proof of Theorem~\ref{thm:upper}}

Assume  $f\in B^{s,0}_{p,\infty}$ and $\| f \|_{B^{s,0}_{p,\infty}} \leq M$.
Then by Theorem~\ref{thm:Besov},
\begin{equation}\label{Aj<M}
\|A_j  f -f  \|_p \leq  c\|f\|_{  B^{s,0}_{p,\infty}}2^{-js}
\le cM2^{-js}.
\end{equation}
Now from
 \[
dY = R f d \mu + \eps dW
\]
we have
\begin{align*}
 \int  g_{k,l,i} \,dY &= \int_Z  R f \, g_{k,l,i} \, d \mu  + \eps \int
 g_{k,l,i}  \,dW = \int_{B^d}   f  \,R^* g_{k,l,i}\,  dx  +  \eps\,
 Z_{k,l,i} \\
&=   \lambda_k  \int_{B^d}   f \, f_{k,l,i} \, dx  +  \eps \, Z_{k,l,i}
\end{align*}
and hence
\[
\frac 1{\lambda_k }  \int  g_{k,l,i} \, dY
= \int_{B^d}   f\, f_{k,l,i}dx  +  \frac{\eps}{\lambda_k} \, Z_{k,l,i}.
\]
On account of (\ref{estimator-alpha}) this leads to
\begin{align*}
\widehat{\alpha}_{j, \xi}
&=\sum_{k, l, i} \gamma^{j,\xi}_{k,l,i}\int_{B^d}ff_{k,l,i}dx
+  \sum_{k, l, i}\gamma^{j,\xi}_{k,l,i}\frac{\eps}{\lambda_k} \, Z_{k,l,i}\\
&=\alpha_{j, \xi} + Z_{j,\xi}.
 \end{align*}
 Here the summation is over
$\{(k,l,i): 0\le k < 2^j, 0\le l \le k, l\equiv k (\mod 2) ,1 \leq i \leq N_{d-1}(l)\}$.
Since $Z_{k,l,i} $ are independent $N(0,1)$ random variables,
$ Z_{j,\xi} \sim  N(0,\sigma^2_{j,\xi})$
with
\begin{equation}\label{sigma-j}
\sigma^2_{j,\xi}
= \eps^2 \sum_{k, l, i}
|\gamma^{j,\xi}_{k,l,i}|^2 \frac{(k)_d}{\pi^{d-1}2^d k}
\le \frac{(2^j)_{d-1}}{ \pi^{d-1} 2^d}
\leq c2^{j(d-1)} \eps^2
\end{equation}
with $c=(d/2\pi)^{d-1}$.
Here we used that $\{f_{k,l,i}\}$ is an orthonormal basis for $\bL^2$ and hence
$\sum_{k, l, i} |\gamma^{j,\xi}_{k,l,i}|^2 = \|\phi_{j, \xi}\|_2^2\le 1$.


From (\ref{estimator-f})
$\widehat{f}_j = \sum_{\xi \in \chi_j}\widehat{\alpha}_{j,\xi} \phi_{j,\xi} $and using (\ref{Aj<M}) we have, whenever $1 \leq p<\infty$,
\begin{align}\label{est-f-fj-p}
\bE  \| f- \widehat{f_j} \|_p^p &\leq 2^{p-1}
\{\| f- A_jf\|_p^p  + \bE \| A_jf- \widehat{f_j} \|_p^p \}\notag\\
&\leq 2^{p-1}\{cM^p2^{-jsp}  + \bE \| A_jf- \widehat{f_j} \|_p^p \}
\end{align}
and, for $p=\infty$,
\begin{align}\label{est-f-fj-infty}
\bE  \| f- \widehat{f_j} \|_\infty
&\leq \| f- A_jf\|_\infty + \bE \| A_jf- \widehat{f_j} \|_\infty\notag\\
&\leq cM2^{-js}  + \bE \| A_jf- \widehat{f_j} \|_\infty.
\end{align}
On the other hand, using inequality (\ref{B2prime}) of Theorem~\ref{thm:stability}
we obtain, if $1 \leq p<\infty$,
\[
\|A_jf-\widehat{f}_j\|_p^p
= \|\sum_{\xi \in \chi_j} (\alpha_{j,\xi}  - \widehat{\alpha}_{j,\xi}) \phi_{j,\xi}\|_p^p
\leq c\sum_{\xi \in \chi_j} |\alpha_{j,\xi} - \widehat{\alpha}_{j,\xi}|^p
\| \phi_{j,\xi}\|_p^p
\]
and hence
\begin{equation}\label{biaisvariance}
\bE \| A_jf- \widehat{f_j} \|_p^p
\leq c\sum_{\xi \in \chi_j}  \bE |Z_{j,\xi}|^p
\| \phi_{j,\xi}\|_p^p \leq c(\eps 2^{j(d-1)/2})^p\sum_{\xi \in \chi_j}
\| \phi_{j,\xi}  \|_p^p,
\end{equation}
where we used that
$\bE |Z_{j,\xi}|^p\leq c(\eps 2^{j(d-1)/2})^p$.
Similarly, for $p=\infty$,
\[
\| A_jf- \widehat{f_j} \|_\infty
= \|\sum_{\xi \in \chi_j} (\alpha_{j,\xi} - \widehat{\alpha}_{j,\xi})\phi_{j,\xi}\|_\infty
\leq c\max_{\xi \in \chi_j}|\alpha_{j,\xi} - \widehat{\alpha}_{j,\xi}|\| \phi_{j,\xi}  \|_\infty
\]
and hence
\begin{align}
\bE \| A_jf- \widehat{f_j} \|_\infty
&\leq c\bE \{\max_{\xi \in \chi_j}
|Z_{j,\xi}    |  \| \phi_{j,\xi}  \|_\infty  \} \nonumber
\\
& \leq c\eps 2^{j(d-1)/2}\max_{\xi \in \chi_j}
\|\phi_{j,\xi}  \|_\infty \sqrt{2\log_2 2^{jd}}
\leq c\eps 2^{jd}\sqrt{j}. \label{biaisvarianceinfty}
\end{align}
For the second inequality above we used  Pisier's lemma:
If $Z_j \sim N(0,\sigma^2_j)$, $\sigma_j \leq \sigma$, then
\[
\bE(\sup_{1\leq j \leq N} |Z_j| ) \leq \sigma \sqrt{2 \log_2N}.
\]
We also used that
$\max_{\xi \in \chi_j} \|\phi_{j,\xi} \|_\infty  \leq c2^{j(d+1)/2}$,
which follows by inequality (\ref{B}) of Theorem~\ref{thm:stability}.


Combining (\ref{est-f-fj-infty}) and (\ref{biaisvarianceinfty})
we obtain, for $p=\infty$,
\[
\bE  \| f- \widehat{f_j} \|_\infty
\leq  c\{M2^{-js }  +  \eps 2^{jd}   \sqrt{j}\}
\]
and if $M2^{-j(s+d)} \sim \eps$, then
\[
\bE  \| f- \widehat{f_j} \|_\infty
\leq c M^{\frac d{s+d}}\eps^{\frac  s{s+d}}\sqrt{\log M/\eps}.
\]
Similarly, combining estimate (\ref{B}) of Theorem~\ref{thm:stability} with (\ref{biaisvariance})
and inserting the resulting estimate in (\ref{est-f-fj-p})
we obtain in the case $4\leq p<\infty$
\begin{align*}
\bE  \| f- \widehat{f_j} \|_p^p
&\leq c\{M2^{-jsp} +(\eps 2^{j(d-1)/2})^p 2^{j dp/2 + p/2 - 2} \}\\
&= c\{M^p2^{-jsp} +\eps^p 2^{j(dp-2)}\}.
\end{align*}
If $M2^{-js}  \sim \eps 2^{j(d-2/p)}$ this yields
\[
\bE \|f- \widehat{f_j} \|_p^p \leq cM^{\frac{(d-2/p)p}{s+d-2/p}}\eps^{\frac{sp}{s+d-2/p}}.
\]
Accordingly, for $p=4$ we combine inequality (\ref{BI}) with (\ref{biaisvariance})
and insert the result in (\ref{est-f-fj-p}) to obtain
\begin{align*}
\bE \| f- \widehat{f_j} \|_p^p
&\leq c\{M^p 2^{-jsp} +(\eps 2^{j(d-1)/2})^p  j 2^{j dp/2}\}\\
&= c\{M^p2^{-jsp} +j (\eps 2^{j(d-1/2})^p\}
\end{align*}
and if $M 2^{-js}  \sim \eps 2^{j(d-1/2)}$ this yields
\[
\bE  \| f- \widehat{f_j} \|_p^p \leq cM^{ \frac{(d-2/p)p}{s+d-2/p}}\eps^{\frac{sp}{s+d-1/2}}\log{M/\eps}.
\]
Finally, if  $1\leq p<4$ as above we obtain using (\ref{B}), (\ref{biaisvariance}),
and (\ref{est-f-fj-p})
\begin{align*}
\bE  \| f- \widehat{f_j} \|_p^p
&\leq c\{M^p2^{-jsp} +(\eps 2^{j(d-1)/2})^p  2^{j dp/2    } \}\\
& = c\{M^p2^{-jsp} + (\eps 2^{j(d-1/2})^p\}.
\end{align*}
So, if
$M2^{-js}  \sim \eps 2^{j(d-1/p)}$, then
\[
\bE  \| f- \widehat{f_j} \|_p^p \leq c M^{\frac{(d-1/2)p}{s+d-1/2}}
\eps^{\frac{sp}{s+d-1/2}}.
\]
This completes the proof of the theorem.
\qed

\section{Application to the Fan Beam Tomography}

\setcounter{equation}{0}

\subsection{Radon and $2d$ Fan Beam Tomography}

We have implemented this scheme for  $d=2$ in the radiological
setting of Cormack\cite{cormack64:_repres_ii}.
This case  corresponds to the fan beam Radon transform used in
Computed Axial Tomography (CAT). As shown if 
Figure~\ref{fig:CAT}, an object is positioned in the middle of the
device.
X rays are sent from a pointwise source $S(\theta_1)$ located on the
boundary and making an angle $\theta_1$ with the horizontal.
They go through the object and are 
received on the other side on
uniformly sampled array of receptors
$R(\theta_1,\theta_2)$.
The log decay of the energy from the source to a receptor is proportional
to the integral of the density $f$ of the object along the ray and thus one
finally measures
\[
\tilde Rf(\theta_1,\theta_2)  = \int_{e_{\theta_1} + \lambda
    e_{\theta_1-\theta_2} \in B^2
} f(x) d\lambda
\]
with $e_\theta=(\cos \theta,\sin \theta)$
or equivalently the classical Radon transform
\[
Rf(\theta, s) =
\int_{\substack{y\in\theta^\perp\\
s\theta+y\in B^1}} f(s\theta  + y) dy,
\quad \theta \in \bS^{1}, \;s \in[-1,1],
\]
for $\theta=\theta_1-\theta_2$ and $s=\sin \theta_2$.
The device is
then rotated to a different angle $\theta_1$ and the process is repeated.
Note that $d\theta \frac{ds}{(1-s^2)}$ is nothing but the measure
corresponding to the uniform $d\theta_1 d\theta_2$ by the change of
variable that maps $(\theta_1,\theta_2)$ into $(\theta,s)$.

\begin{figure}
  \centering

\includegraphics[width=8cm]{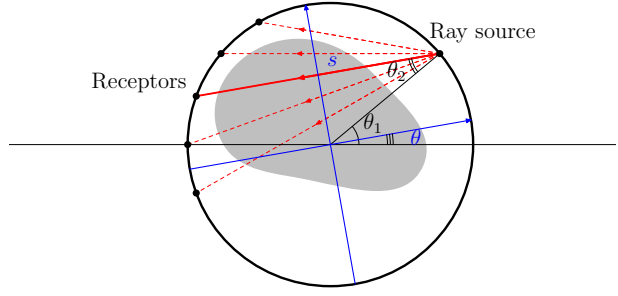}
  
  \caption{Simplified CAT device}
  \label{fig:CAT}
\end{figure}

The Fan Beam Radon SVD basis of the disk is tensorial in polar coordinates:
\[
f_{k,l,i} (r,\theta)
= (2k+2)^{1/2}P_j^{(0,\, l)} (2|r|^2-1)|r|^l Y_{l, i}(\theta),
\;\; 0\leq l \leq k,\; k-l =2j, \; 1\leq i \leq 2,
\]
where $P_{j}^{0,l}$ is the corresponding Jacobi polynomial, and
 $Y_{l,1}(\theta)=c_l \cos(l\theta)$ and $Y_{l,2}(\theta)=c_l
\sin(\theta)$
with $c_0=\frac{1}{\sqrt{2\pi}}$ and $c_l=\frac{1}{\sqrt{\pi}}$ otherwise.
The basis of $S^2\times[-1,1]$ has a similar tensorial structure as it is given by
\[
g_{k,l,i}(\theta, t)
=  [h_k]^{-1/2}(1-t^2)^{1/2} C^{1}_k(t) Y_{l,i }(\theta),
\quad k\ge 0, \; l \ge 0, \; 1\leq i \leq 2,
\]
where $C^{1}_k$ is the Gegenbauer of parameter $1$ and degree $k$. 
The corresponding eigenvalues are
\[
\lambda_k = \frac{2\sqrt{\pi}}{\sqrt{k+1}}.
\]

In this paper, we have considered the theoretical framework of the
white noise model. In this model, we assume that we have access to the
noisy ``scalar product'' $\int g_{k,l,i}dY$, that is to the scalar product
of $Rf$ with the SVD basis $g_{kl,i}$ up to a i.i.d. centered Gaussian
perturbation of known variance $\epsilon^2$.
This white noise model is a convenient statistical framework closely
related to a more classical 
regression problem with a uniform design on $\theta_1$ and
$\theta_2$, which is closer to the implementation in real devices.
In this regression design, one observe
\[
Y_{i_1,i_2} = Rf\left(2\pi \left(\frac{i_1}{N_1}-\frac{i_2}{N_2}\right), \sin
  2 \pi \frac{i_2}{N_2}\right)+\epsilon_{i_1,i_2},\; i_1\le N_1,\; i_2\le N_2
\]
where $N_1$ and $N_2$ gives the discretization level of the angles
$\theta_1$ and $\theta_2$ and $\epsilon_{i_1,i_2}$ is an
i.i.d. centered Gaussian sequence of known variance $\sigma^2$.
Note that this points are not cubature points for the SVD coefficients.
The correspondence between the two model is obtain by replacing
 the noisy scalar product $\int g_{k,l,i}dY$
with
 by the corresponding Riemann sum
\[
 \widehat{\langle Rf, g_{k,l,i}\rangle}  =\frac{1}{N_1\times N_2} \sum_{i_1=0}^{N_1-1} \sum_{i_2=0}^{N_2-1}
g_{k,l,i}\left(2\pi \left(\frac{i_1}{N_1}-\frac{i_2}{N_2}\right), \sin
  2 \pi \frac{i_2}{N_2}\right) Y_{i_1,i_2}
\]
and using the calibration $\epsilon^2 = \sigma^2/(N_1\times N_2)$.
It is proved, for instance in \cite{brownlow96} that the regression model with uniform design and the white noise model are close in the sense of Le Cam's deficiency -which roughly means that any procedure can be transferred from one model to the other, with the same order of risk-. 
The estimator $\widehat{f}_j$  defined in the white noise model by
\[
\widehat{f}_j  =  \sum_{k,l,i} \frac{a\left(\frac{k}{2^j}\right)}{\lambda_k}\int g_{k,l,i}dY\,
f_{k,l,i} = \sum_{k,l,i} \frac {a\left(\frac{k}{2^j}\right)\sqrt{k+1}}{2\sqrt{\pi}}\int g_{k,l,i}dY\,
f_{k,l,i}
\]
is thus replaced in the regression model by
\[
\widehat{f}_j  =  \sum_{k,l,i}
\frac{a\left(\frac{k}{2^j}\right)}{\lambda_k} \widehat{\langle Rf, g_{k,l,i}\rangle}
\,
f_{k,l,i} = \sum_{k,l,i} \frac
{a\left(\frac{k}{2^j}\right)\sqrt{k+1}}{2\sqrt{\pi}}\widehat{\langle Rf, g_{k,l,i}\rangle}
\,
f_{k,l,i}\quad.
\]

\subsection{Numerical results}

To illustrate the advantages of the linear needlet estimator over the
linear SVD estimator, we have compared their performances on a
synthetic example, the classical Logan Shepp phantom\cite{loganshepp},
 for different $L^p$ norm and different noise level, and for
both the white noise model and the regression model. The Logan
Shepp phantom is a synthetic image used as a benchmark in the tomography
community. It is a simple toy model for human body structures
simplified as a
piecewise constant function with
discontinuities along ellipsoids (see  
Figure~\ref{fig:numcomp}). This example
is not regular in a classical sense.
Indeed, it belongs to $B_{1,1}^{1,0}$
but  not to any $B_{p,q}^{s,0}$ with $s>1$. 

To conduct the experiments, we have adopted the following scheme.
Denote by $f$ of the Logan Shepp function presented above,
its decomposition in the SVD basis $f_{k,l,i}$ up to degree $\tilde k=512$
has been approximated with an initial numerical quadrature $\chi$ valid for polynomial of degree
$4\times \tilde k=2048$,
\[
\langle f, f_{k,l,i} \rangle \simeq \sum_{(r_i,\theta_i)\in\chi}
\lambdaquadra_{(r_i,\theta_i)} f(r_i,\theta_i) f_{k,l,i}(r_i,\theta_i) = c_{k,l,i}
\]
and used this value to approximate the original SVD coefficients of $R(f)$, the
noiseless Radon transform of $f$,
\[
\langle R(f), f_{k,l,i} \rangle \simeq \lambda_k c_{k,l,i}.
\]

\textit{In the white noise setting,} for all $k\leq k_0=256$, a noisy observation $\int g_{k,l,i}dY$ is generated by
\[
\int g_{k,l,i}dY \simeq \lambda_k c_{k,l,i} + \epsilon W_{k,l,i}
\]
where $\epsilon$ is the noise level and $W_{k,l,i}$ a iid sequence of
standard Gaussian random variables.
Our linear needlet estimator $\widehat{f}_J$ of level $J=\log_2(k^N)$ is then computed as
\[
\widehat{f}_J = \sum_{k\leq k_0,l,i}
a\left(\frac{k}{2^J}\right)(c_{k,l,i} + \frac{\epsilon}{\lambda_k}
W_{k,l,i}) f_{k,l,i}
\]
while the linear SVD estimator $\widehat{f}^{S}_{k^S}$ of degree
$k^{S}$ is defined as
\[
\widehat{f}^S_{k^S} = \sum_{k\leq k^S,l,i}
(c_{k,l,i} + \frac{\epsilon}{\lambda_k}
W_{k,l,i}) f_{k,l,i}.
\]
We also consider the naive inversion up to degree $k_0$
$\widehat{f}^{I}$ which is equal to $\widehat{f}^{S}_{k_0}$.
The $L^p$ estimation error is measured by reusing the initial
quadrature formula,
\[
\|f-\widehat{f}\|^p \simeq \sum_{(r_i,\theta_i)\in\chi}
\lambdaquadra_{(r_i,\theta_i)} |f(r_i,\theta_i)-\widehat{f}(r_i,\theta_i)|^p.
\]

\textit{In the regression setting}, we have computed the values of the
Radon transform $Rf$ of
$f$ on a equispaced grid for the angles $\theta_1$ and $\theta_2$
specified  by its sizes $N_1$ and $N_2$ using its SVD decomposition up
to $k=\tilde k=512$. We have then defined the
noisy observation as
\[
Y_{i_1,i_2} = (Rf\left(2\pi \left(\frac{i_1}{N_1}-\frac{i_2}{N_2}\right), \sin
  2 \pi \frac{i_2}{N_2}\right)+\epsilon_{i_1,i_2}
\]
with $\epsilon_{i_1,i_2}$ an
i.i.d. centered Gaussian sequence of known variance $\sigma^2$.
The estimated SVD coefficients are obtained
 through the Riemann sums
\[
 \widehat{\langle Rf, g_{k,l,i}\rangle}  =\frac{1}{N_1\times N_2} \sum_{i_1=0}^{N_1-1} \sum_{i_2=0}^{N_2-1}
g_{k,l,i}\left(2\pi \left(\frac{i_1}{N_1}-\frac{i_2}{N_2}\right), \sin
  2 \pi \frac{i_2}{N_2}\right) Y_{i_1,i_2}\quad.
\]
We plug then these values instead of the $\int g_{k,l,i} dY$ in the
previous estimators.

For each noise level and each norm, the best level $J$ and the best
degree $K$ has been selected as the one minimizing the average error
over 50 realizations of the noise. 
Figure~\ref{fig:decaywhitenoise} displays, in a logarithmic scale, the estimation errors
 \(
 \|f-\hat f\|_p
 \)
in the white noise model
plotted against the logarithm of the noise level $\epsilon$.
It shows that,
except for the very low noise case,
both the linear SVD estimator and the linear Needlet estimators reduce
the error over a naive inversion linear SVD estimate up to the maximal
available degree $k_0$.
They also show that the Needlet estimator outperforms
the SVD estimator 
  in a large majority of cases from the norm point of
view and almost always from the visual point of view as shown in Figure~\ref{fig:numcomp}.
The localization of the needlet also 'localizes'  the errors and thus the
``simple'' smooth regions are much better restored with the needlet estimate than with
the SVD because the errors are essentially concentrated along the edges for the
needlet. Remark that the results obtained for the regression model in
Figure~\ref{fig:decayreg} are similar. We have plotted, in a
logarithmic scale, the estimation errors against 
the logarithm of the equivalent of the noise $\epsilon^2$ in the
regression $\sigma^2/(N_1\times N_2)$ with $N_1=N_2=64$ and various $\sigma^2$.
Observe that the curves are similar as long as $\sigma^2$ is not too
small, i.e. as long as the error due to the noise dominate the error
due to the discretization. As can be seen both analysis do agree. This confirms the fact that
the white noise model analysis is relevant for the corresponding fixed design.

\begin{figure}
  \centering
  
  \begin{tabular}{ccc}
\includegraphics[width=5cm]{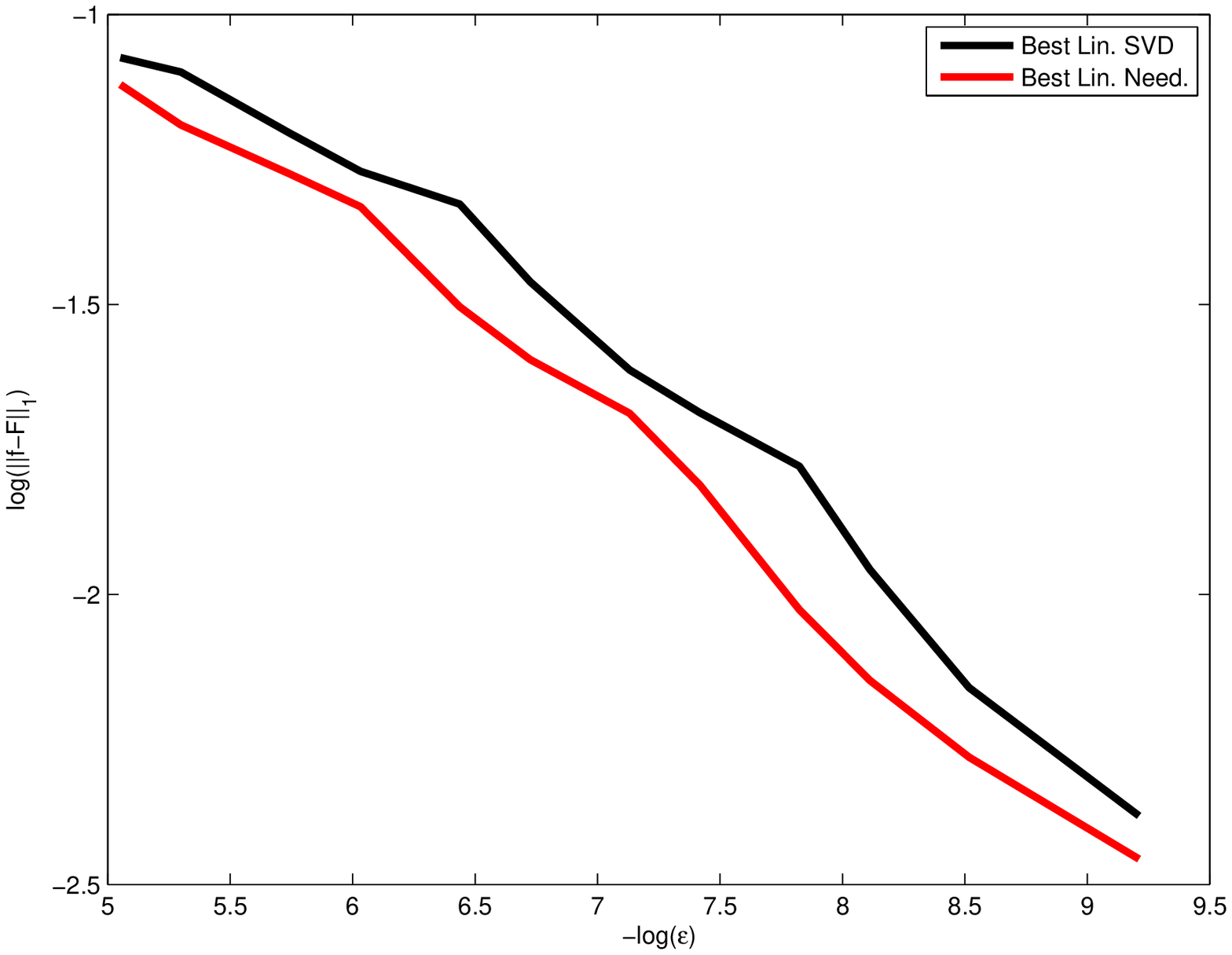}
&&\includegraphics[width=5cm]{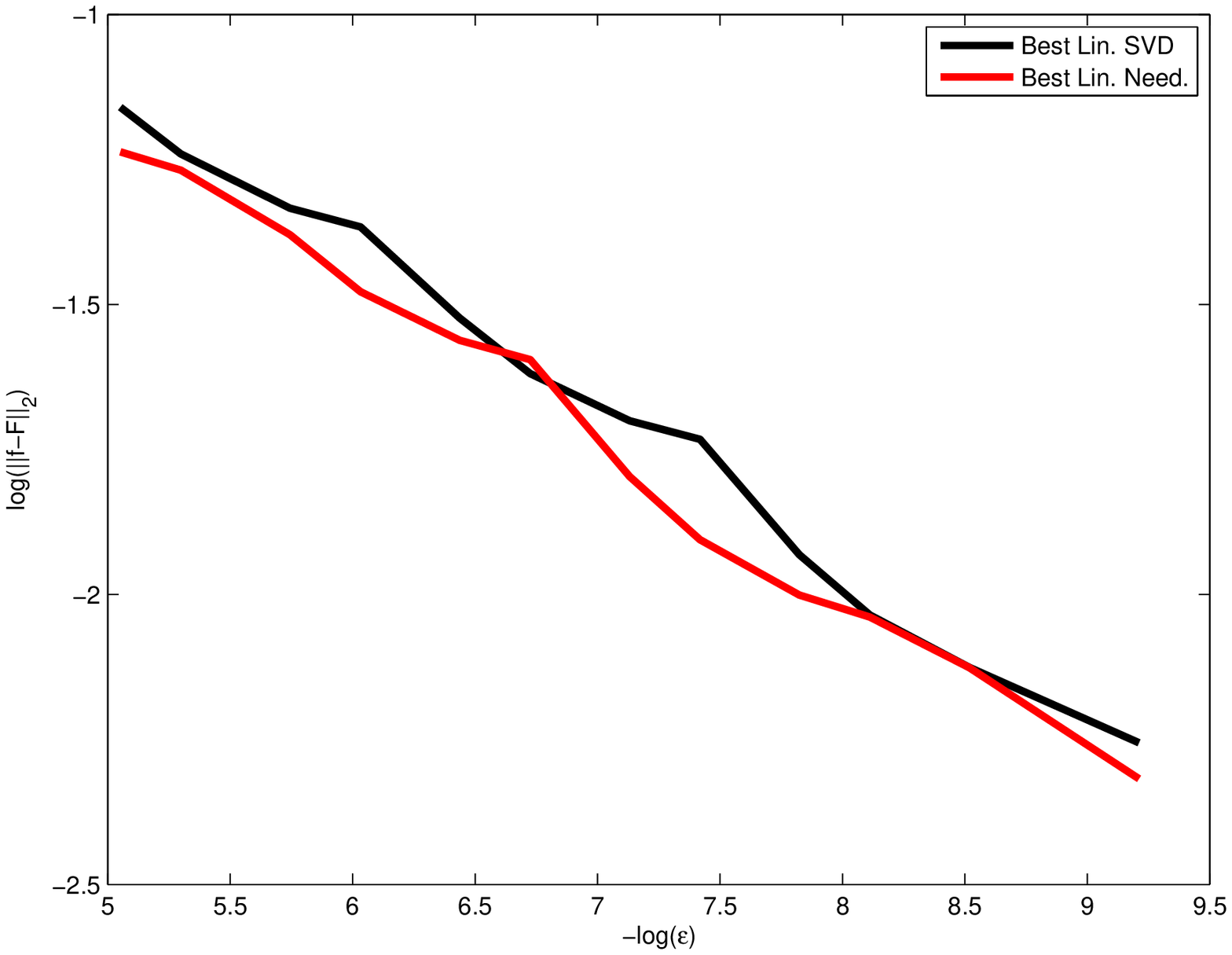}\\
$L^1$ && $L^2$\\
\includegraphics[width=5cm]{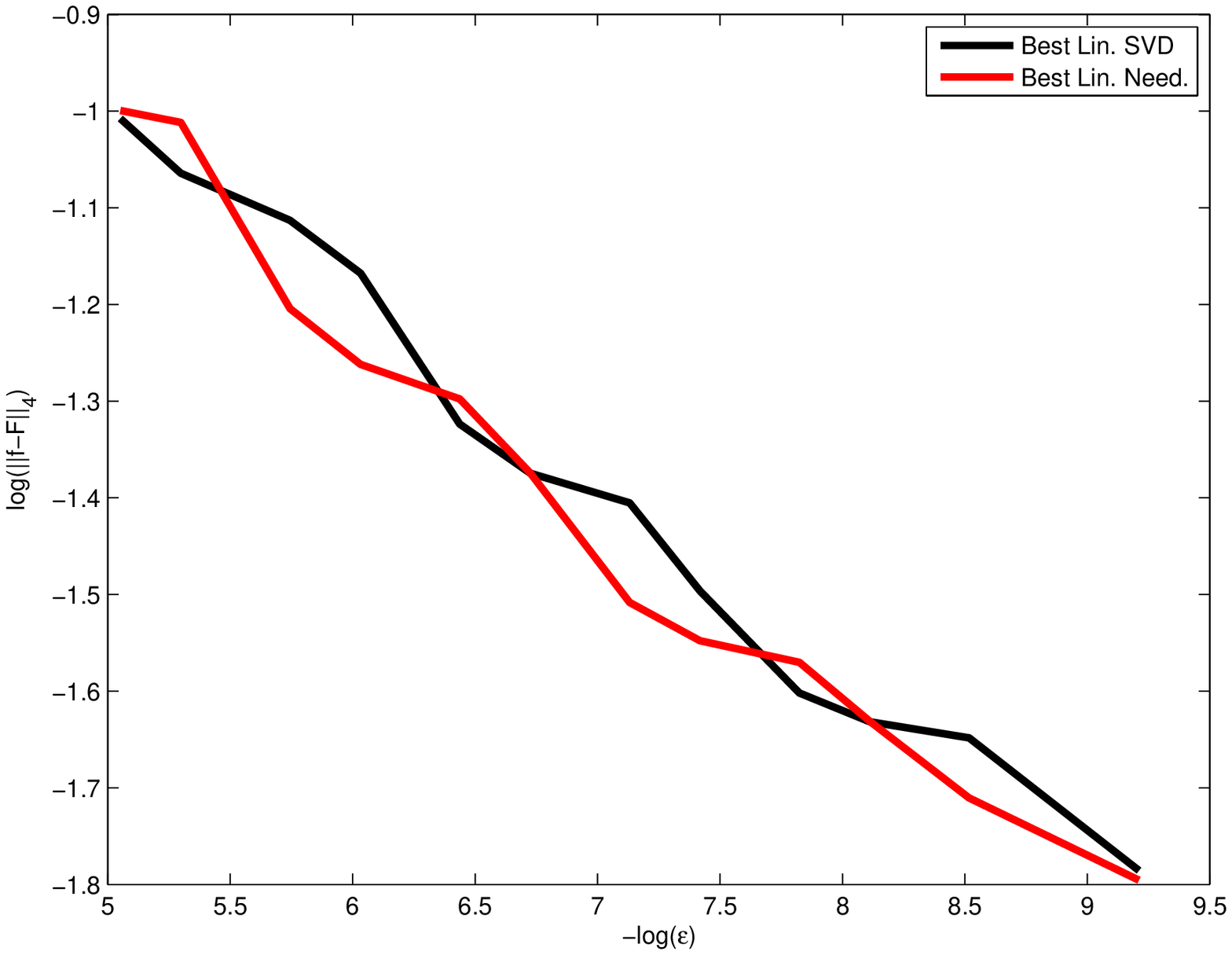}
&&
\includegraphics[width=5cm]{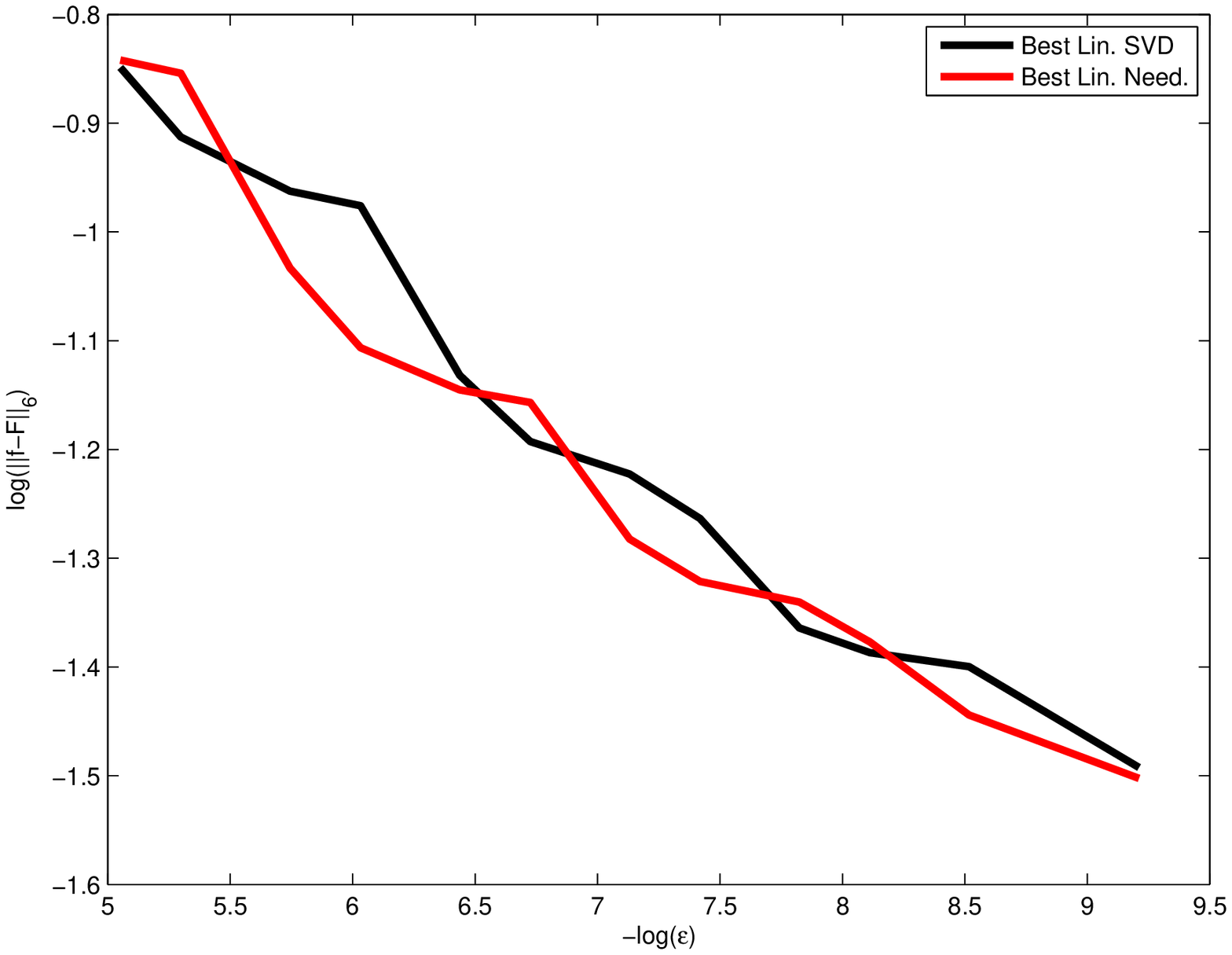}\\
 $L^4$&&$L^6$\\
\includegraphics[width=5cm]{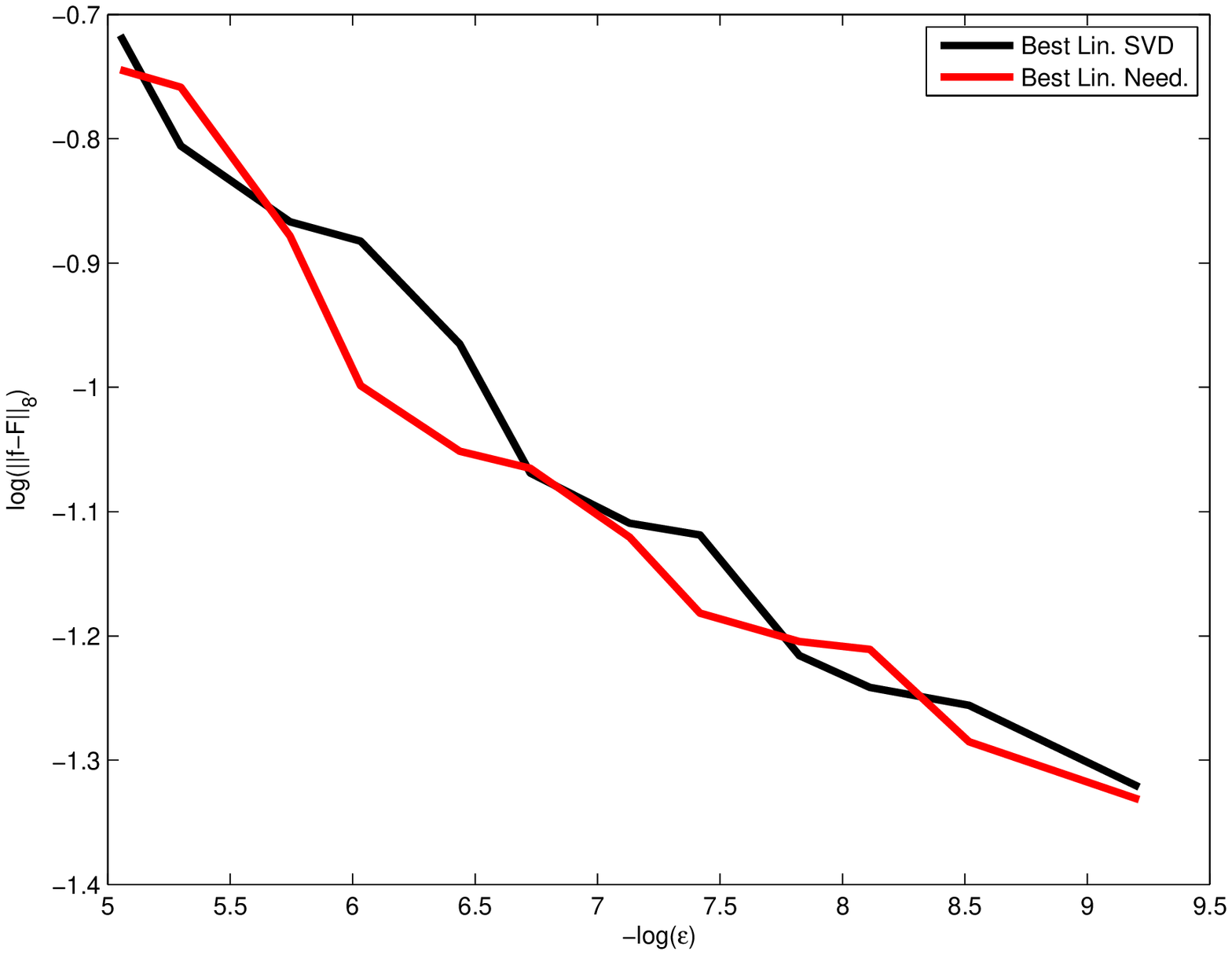}
&&\includegraphics[width=5cm]{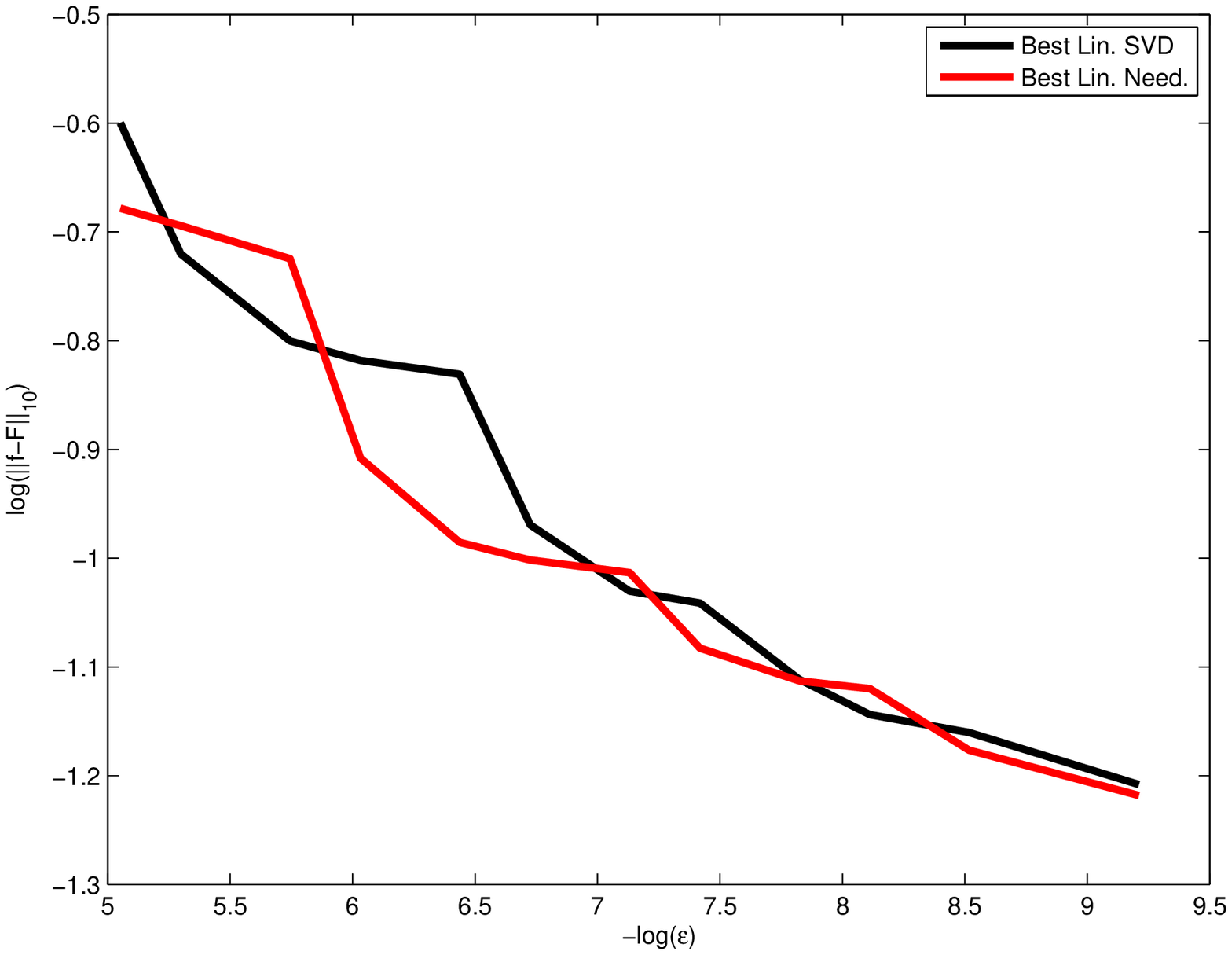}\\
$L^8$ && $L^{10}$\\
&&\includegraphics[width=5cm]{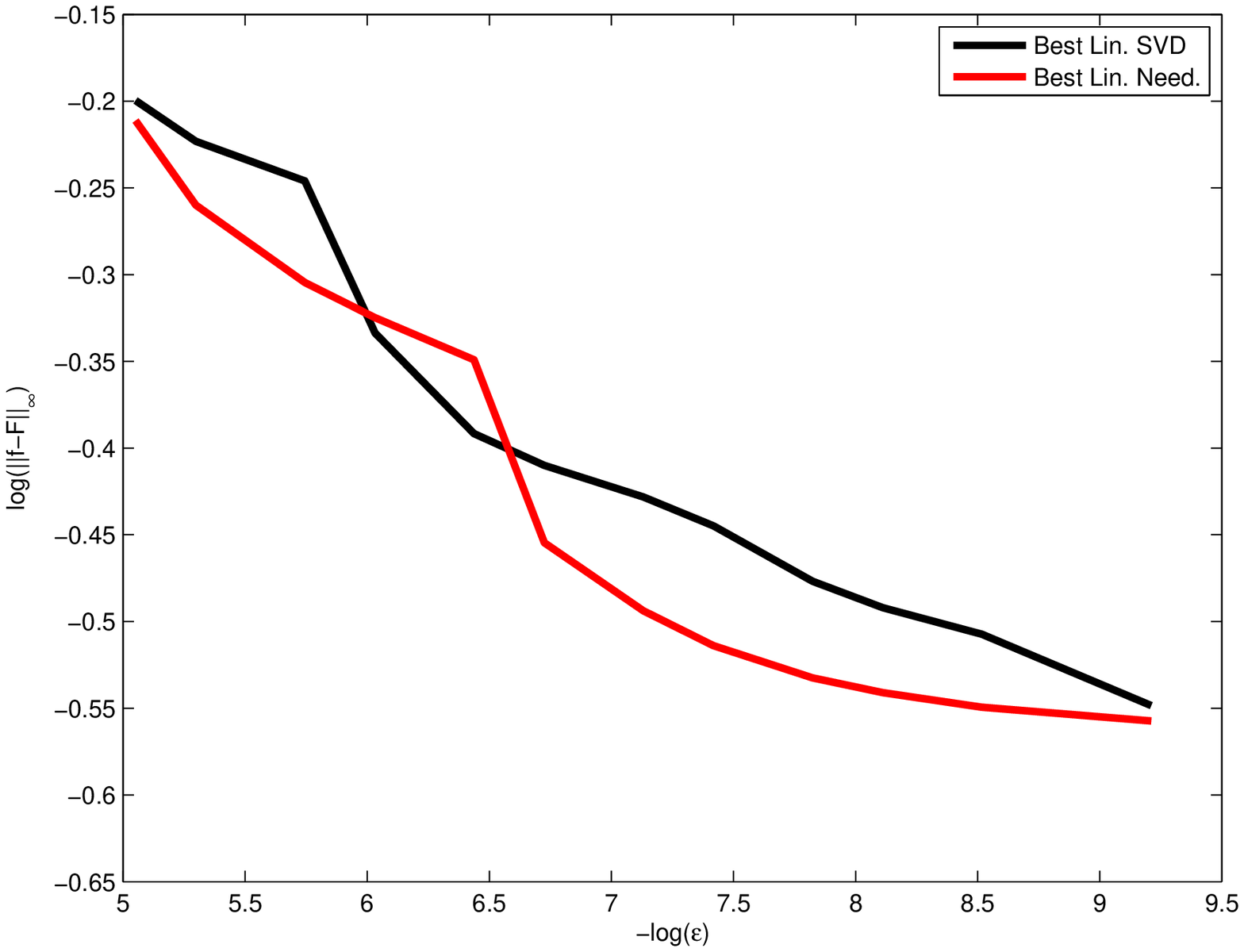}\\
&& $L^{\infty}$
  \end{tabular}
  \caption{Error decay in the white noise model: the red curve
    corresponds to the needlet estimator and the black one to the SVD estimator.}
  \label{fig:decaywhitenoise}
\end{figure}

\begin{figure}
  \centering
  
  \begin{tabular}{ccc}
\includegraphics[width=5cm]{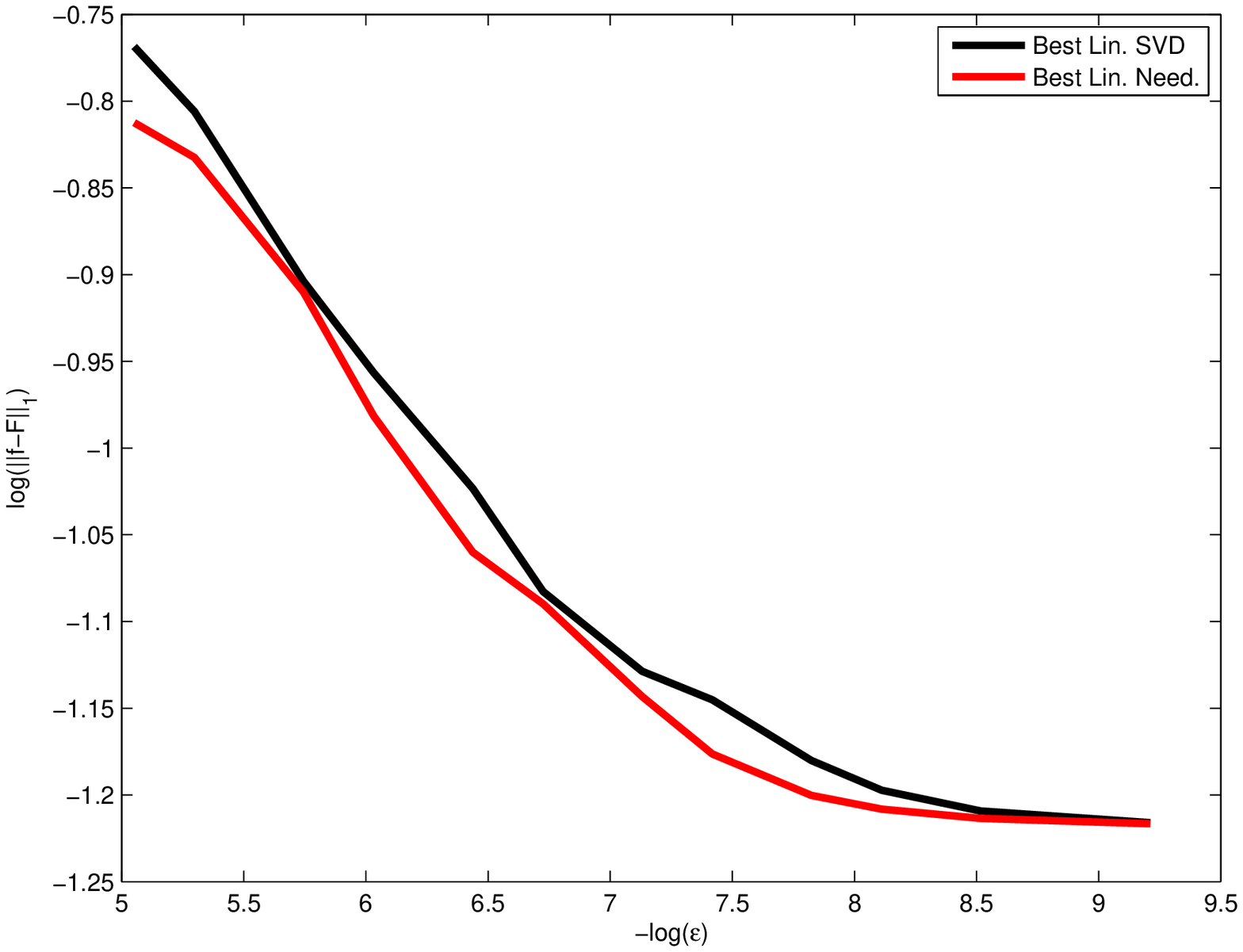}
&&\includegraphics[width=5cm]{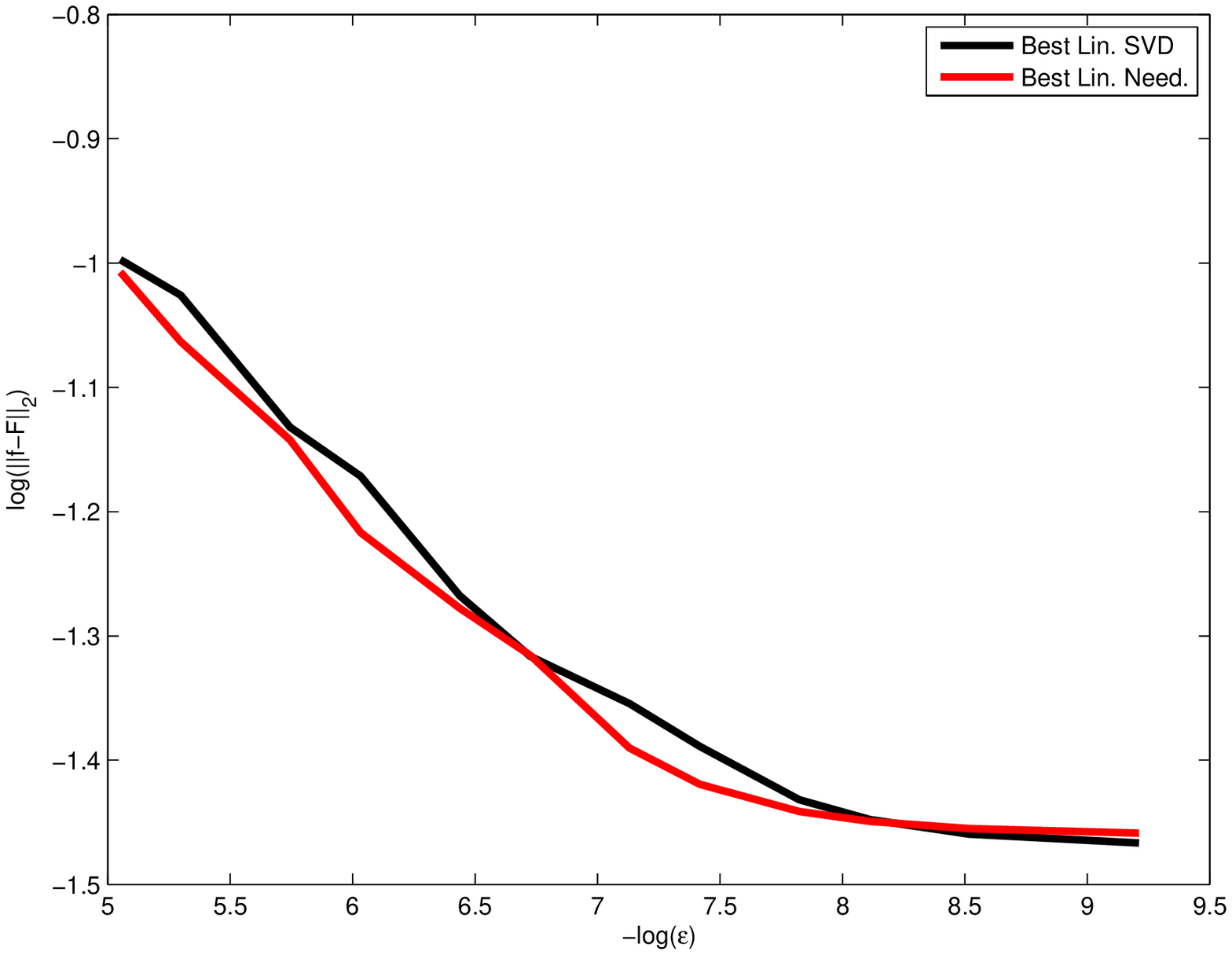}\\
$L^1$ && $L^2$\\
\includegraphics[width=5cm]{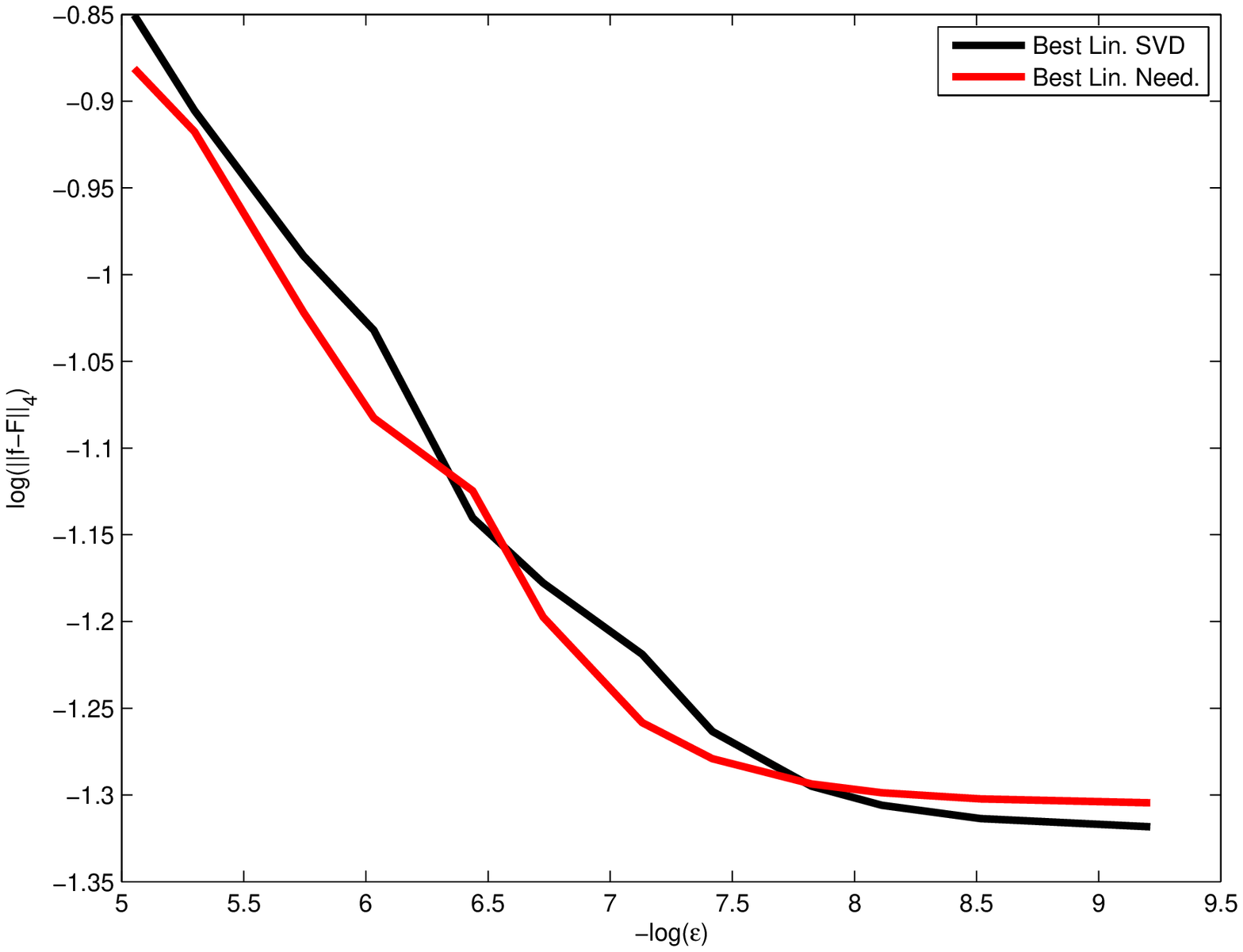}
&&
\includegraphics[width=5cm]{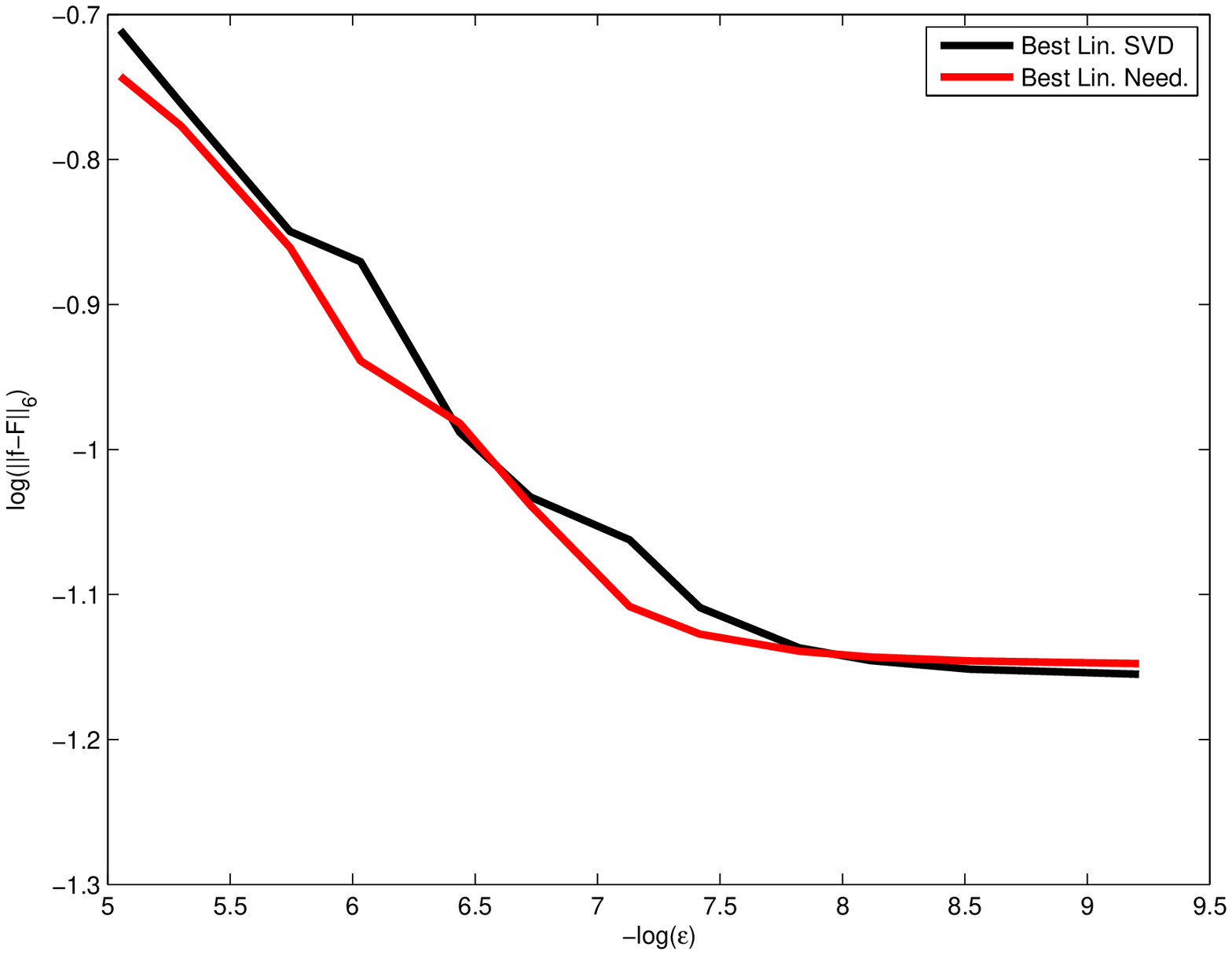}\\
 $L^4$&&$L^6$\\
\includegraphics[width=5cm]{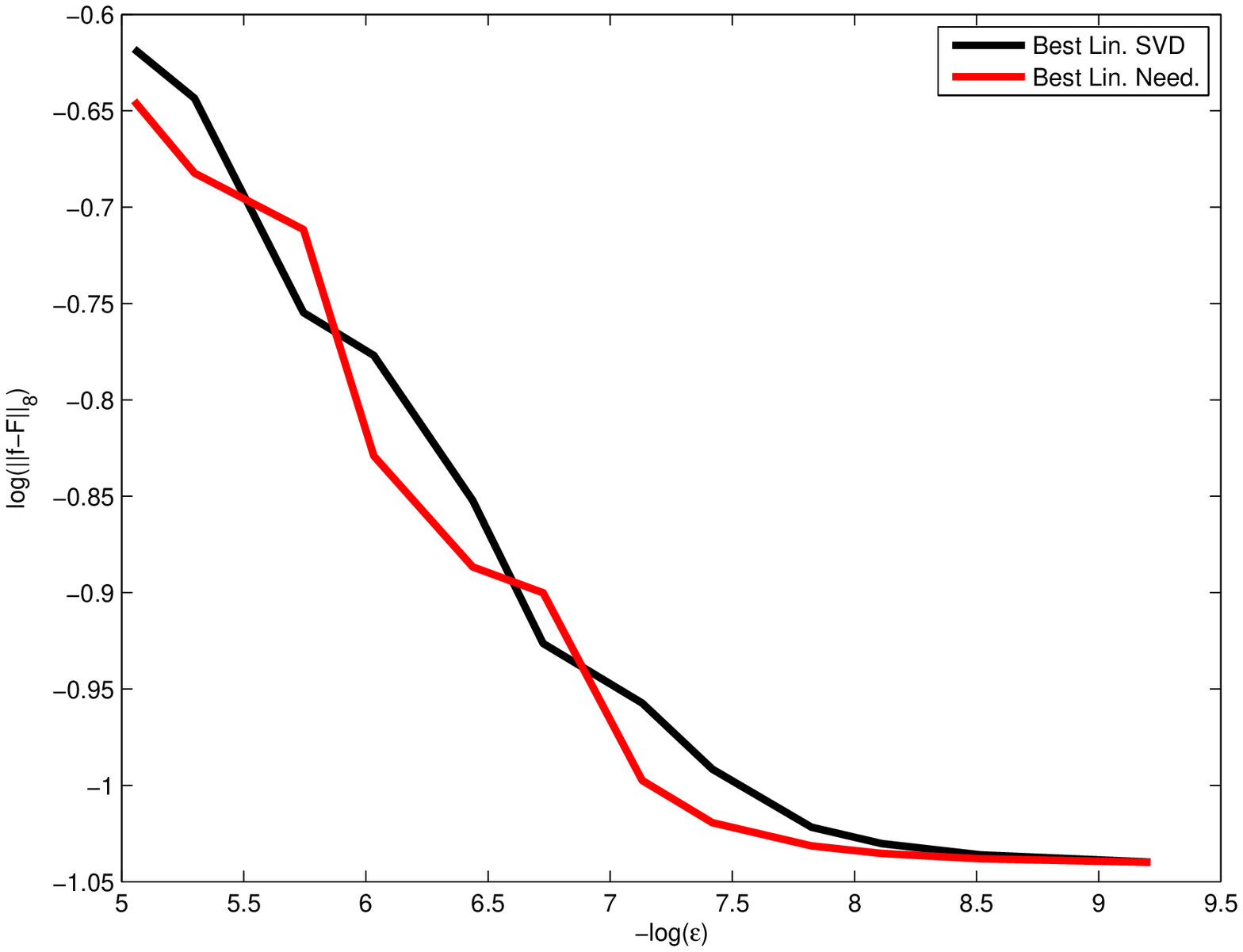}
&&\includegraphics[width=5cm]{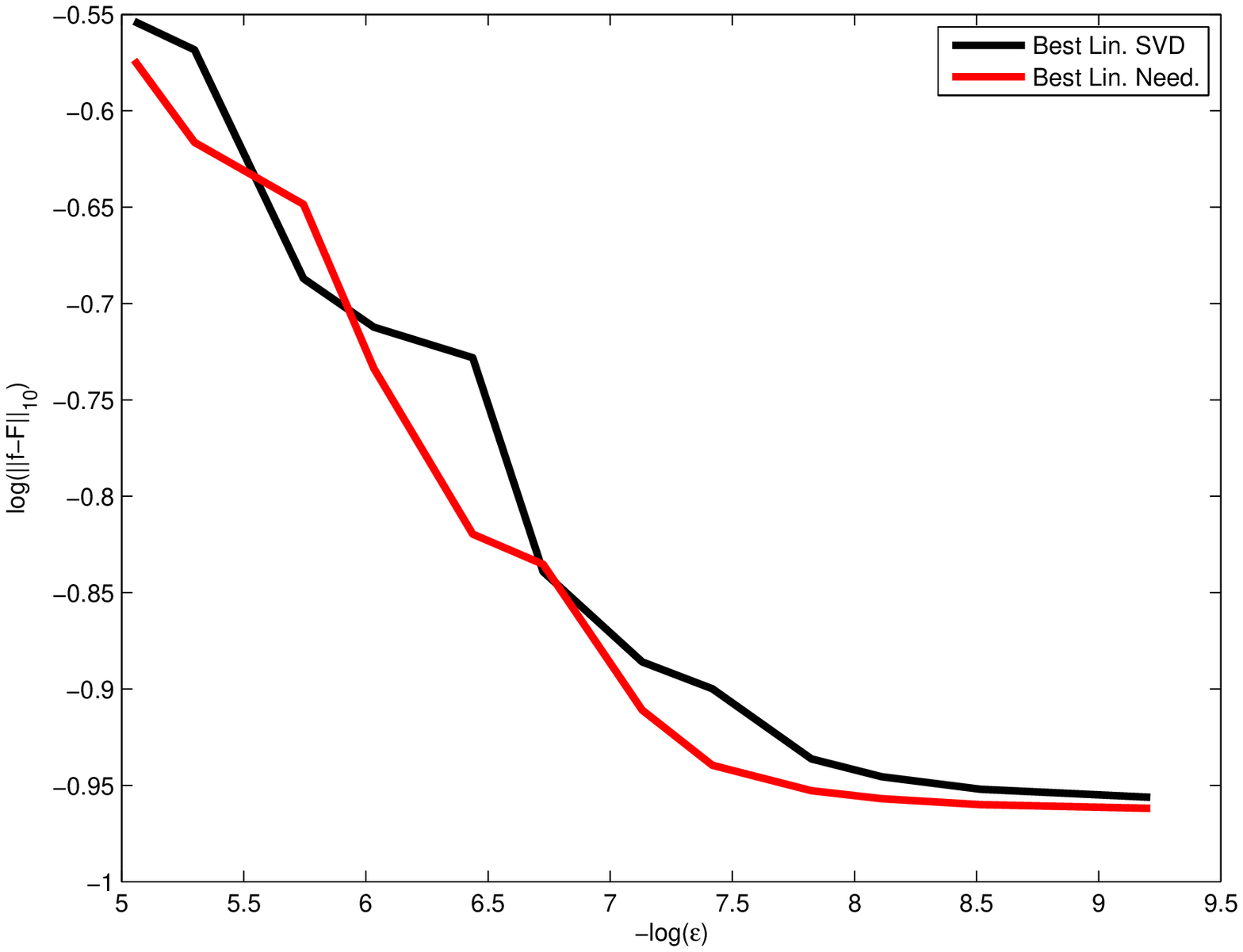}\\
$L^8$ && $L^{10}$\\
&&\includegraphics[width=5cm]{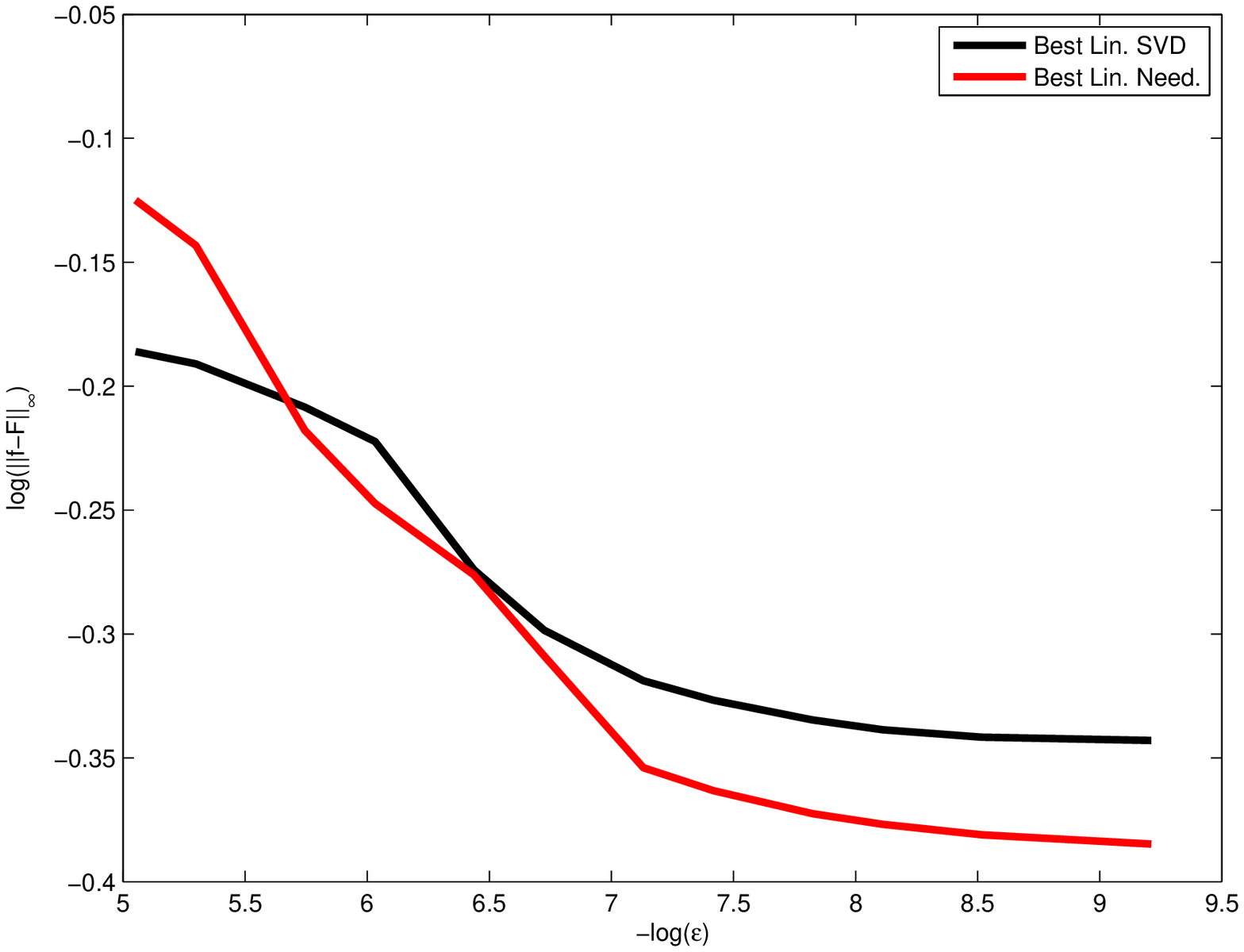}\\
&& $L^{\infty}$
  \end{tabular}
  \caption{Error decay in the regression model:  the red curve
    corresponds to the needlet estimator and the black one to the SVD estimator.}
  \label{fig:decayreg}
\end{figure}

A fine tuning for the choice  of the maximum degree is very important to obtain
a good estimator. In our proposed scheme, and in the Theorem, this
parameter is set by the user according to some expected
properties of the unknown function or using some oracle.
 Nevertheless, an adaptive estimator, which
does not require this input, can already be obtained from this family, for
example, by using some aggregation technique. A different way to
obtain an adaptive estimator based on thresholding is under investigation
by some of the authors.

\begin{figure}
  \centering

\begin{tabular}{ccc}
Original ($f$) && Inversion ($\hat f^I$)\\
\parbox[c]{6cm}{\includegraphics[width=6cm]{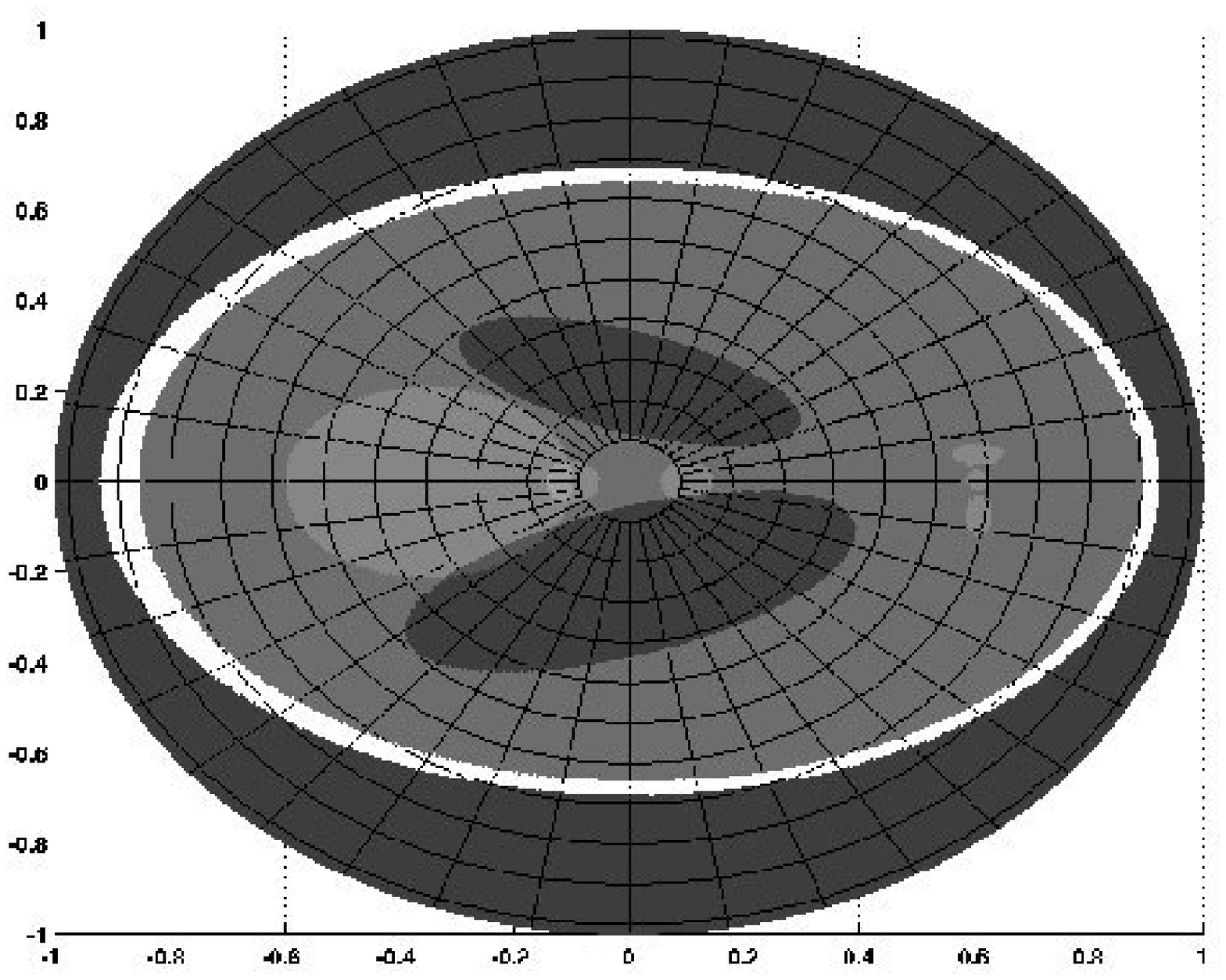}}
&
\hspace*{.5cm}
&
\parbox[c]{6cm}{\includegraphics[width=6cm]{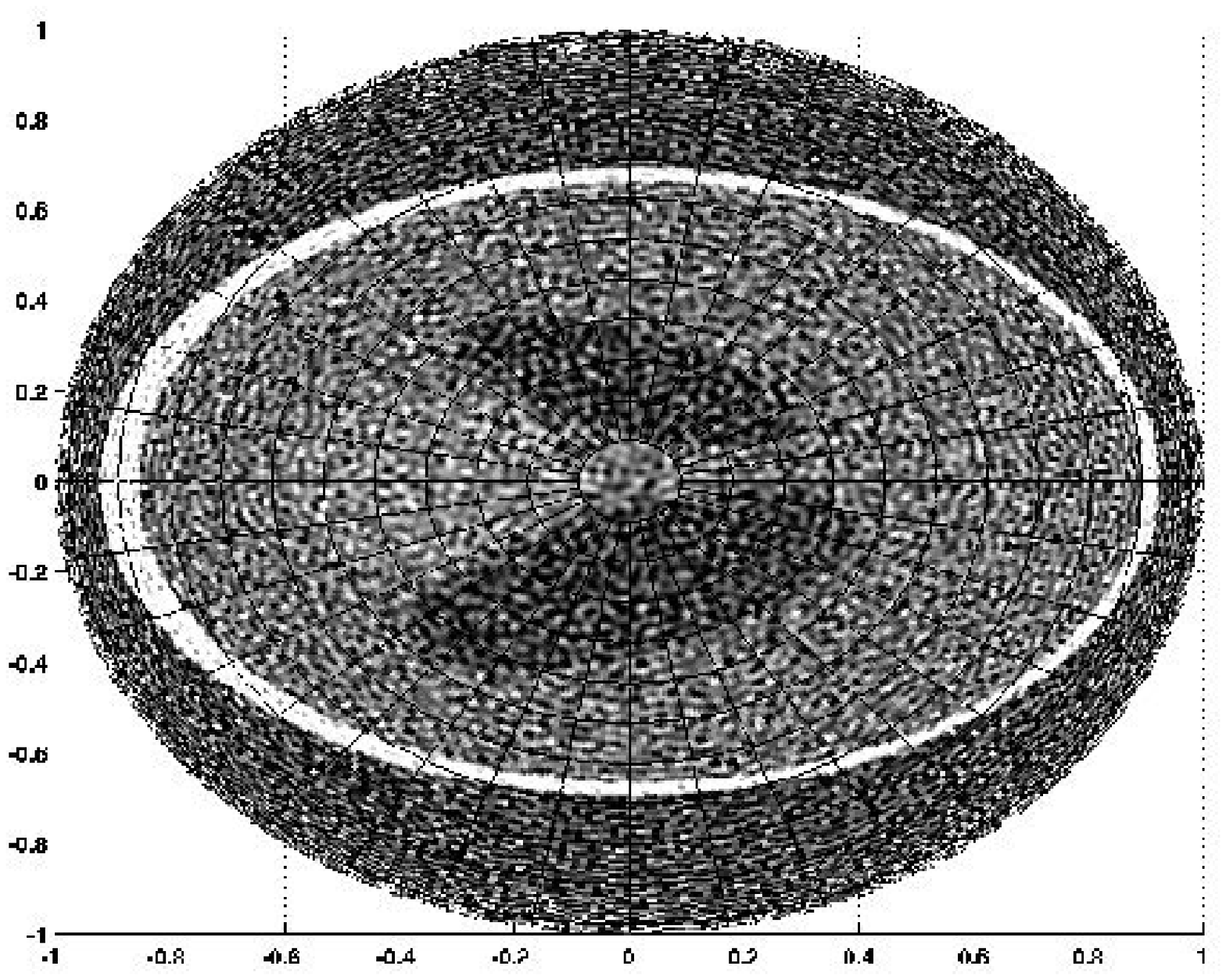}}
\\
\\
\parbox[c]{6cm}{\includegraphics[width=6cm]{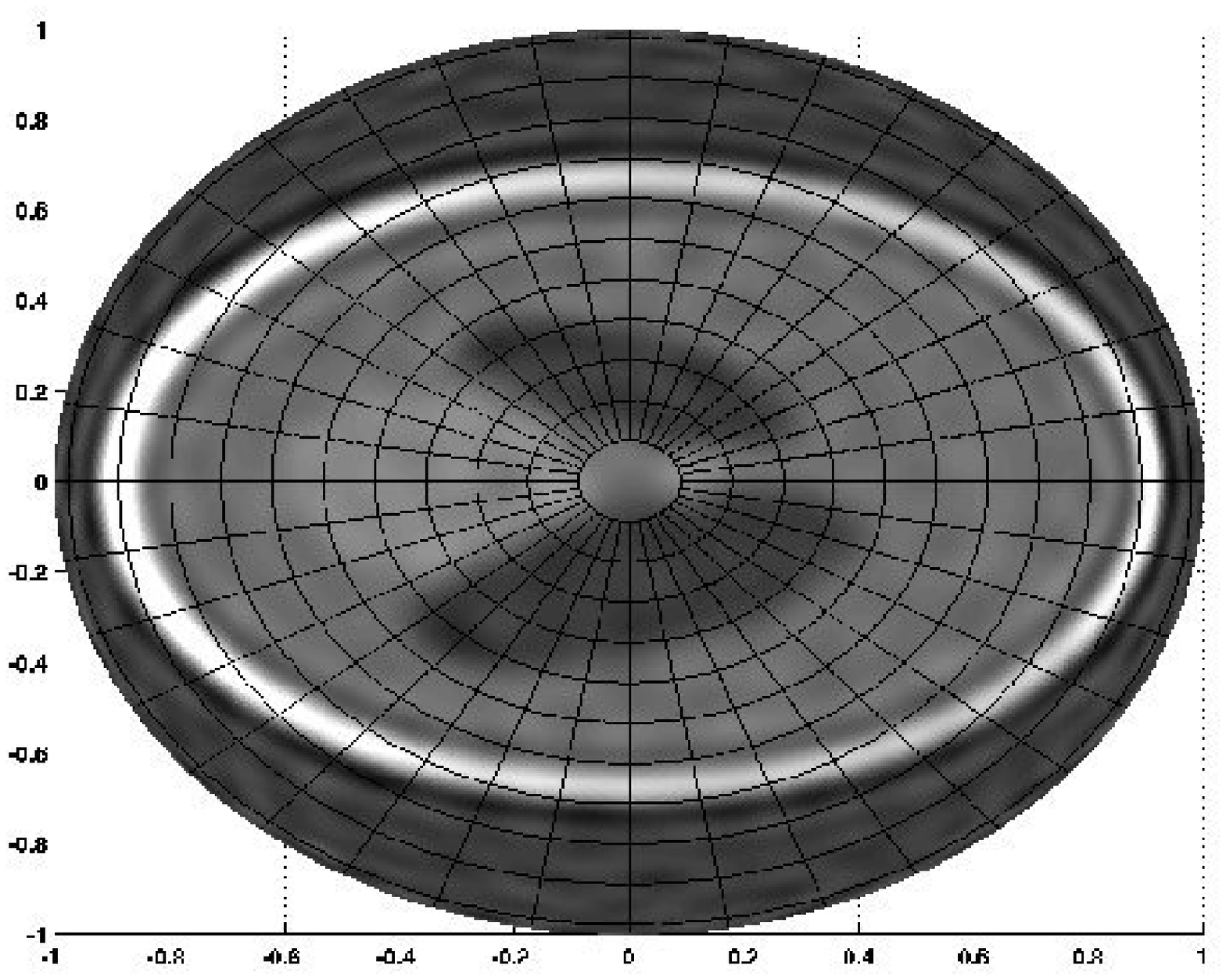}}
&
&
\parbox[c]{6cm}{\includegraphics[width=6cm]{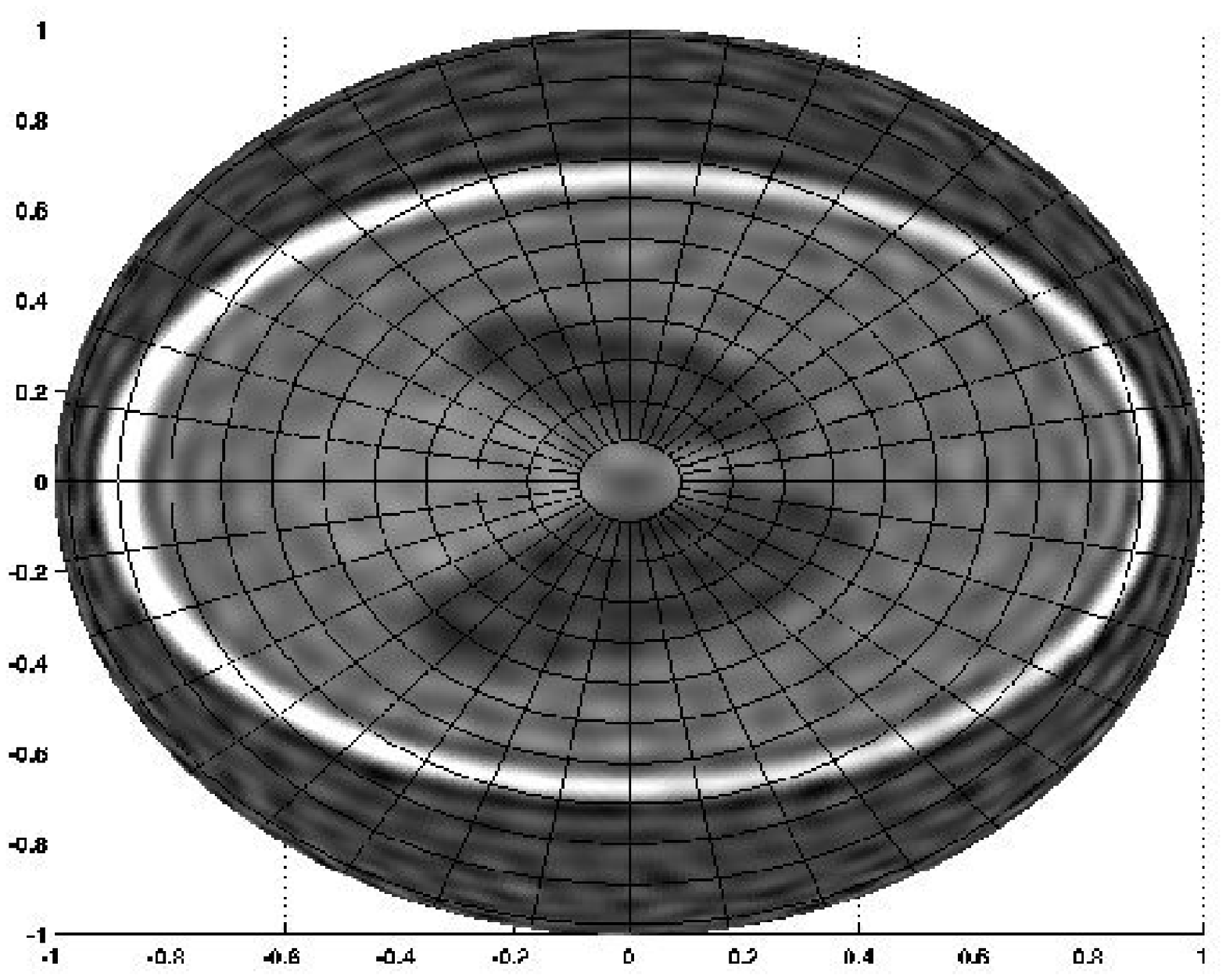}}\\
Needlet ($\hat f^N_{2^6}$)&& SVD ($\hat f^D_{2^5}$)
\end{tabular}  

  \caption{Visual comparison for the original Logan Shepp phantom with
  $\epsilon=8$ (Low quality figure due to
    arXiv constraint)}
  \label{fig:numcomp}
\end{figure}

\section{Appendix}

\subsection{Proof of identity (\ref{lambda-k})}

From \cite[p. 99]{NATT} with some adjustment of notation, we have
$$
\lambda_k^2=\frac{|\bS^{d-2}| \pi^{(d-1)/2}}{\Gamma\big(\frac{d+1}{2}\big) C_k^{d/2}(1)C_l^{(d-2)/2}(1)}
\int_{-1}^1 C_k^{d/2}(t)C_l^{(d-2)/2}(t)(1-t^2)^{(d-3)/2}dt,
$$
where $0\le l\le k$ and $l\equiv k$ $(\mod 2)$.
As will be seen shortly $\lambda_k$ is independent of $l$.

We will only consider the case $d>2$ (the case $d=2$ is simpler, see \cite[p. 99]{NATT}).
To~compute the above integral we will use the well known identity (cf. \cite[(4.7.29)]{SZG})
$$
(n+\lambda)C_n^{\lambda}(t)= \lambda(C_n^{\lambda+1}(t)-C_{n-2}^{\lambda+1}(t)).
$$
Summing up these identities (with indices $n, n-2, \dots$)
and taking into account that
$C_0^\lambda(t)=1$, $C_1^\lambda(t)=2\lambda(t)$, we get
\begin{equation}\label{sum-C}
C_n^{\lambda+1}(t)
= \sum_{j=0}^{\lfloor n/2\rfloor} \frac{n-2j+\lambda}{\lambda}C_{n-2j}^\lambda(t).
\end{equation}
This with $\lambda=(d-2)/2$ and the orthogonality of the polynomials
$C_n^{(d-2)/2}(t)$, $n\ge 0$, yield
\begin{align*}
\int_{-1}^1 C_k^{d/2}(t)C_l^{(d-2)/2}(t)(1-t^2)^{(d-3)/2}dt
&=\frac{l+\lambda}{\lambda}  
\int_{-1}^1 [C_l^{(d-2)/2}(t)]^2(1-t^2)^{(d-3)/2}dt\\
&=\frac{l+\lambda}{\lambda}h_l^{(\lambda)}
=\frac{(l+\lambda)2^{1-2\lambda}\pi}{\lambda\Gamma(\lambda)^2}
\frac{\Gamma(l+2\lambda)}{(l+\lambda)\Gamma(l+1)}.
\end{align*}
We use this and that $C_n^{\lambda}(1)=\frac{\Gamma(n+2\lambda}{n!\Gamma(2\lambda)}$
(see \S\ref{Jacobi}) and $|\bS^{d-2}|=\frac{2\pi^{(d-1)/2}}{\Gamma((d-1)/2)}$ to obtain
\begin{align}\label{express-lambda-k}
\lambda_k^2
&=\frac{2\pi^{d-1}}{\Gamma\big(\frac{d-1}{2}\big)\Gamma\big(\frac{d+1}{2}\big)}
\frac{k!\Gamma(d)l!\Gamma(d-2)}{\Gamma(k+d)\Gamma(l+d-2)}
\frac{2^{4-d}\pi}{(d-2)\Gamma\big(\frac{d-2}{2}\big)^2}
\frac{\Gamma(l+d-2)}{\Gamma(l+1)} \notag\\
&=\frac{2^{5-d}\pi^{d}}{\Gamma\big(\frac{d-1}{2}\big)\Gamma\big(\frac{d+1}{2}\big)}
\frac{\Gamma(d)\Gamma(d-2)}{(d-2)\Gamma\big(\frac{d-2}{2}\big)^2}
\frac{1}{(k+1)_{d-1}}.
\end{align}
The doubling formula for Gamma-function says:
$\Gamma(2z)=\frac{2^{2z-1}}{\sqrt{\pi}}\Gamma(z)\Gamma(z+1)$ (see e.g. \cite{SZG})
and hence
$$
\Gamma(d)\Gamma(d-2)=(d-1)(d-2)\Gamma(d-2)^2
=\frac{2^{2(d-3)}}{\pi}(d-1)(d-2)\Gamma\Big(\frac{d-2}{2}\Big)^2\Gamma\Big(\frac{d-1}{2}\Big)^2.
$$
We insert this in (\ref{express-lambda-k}) and then a little algebra shows that
$\lambda_k^2=\frac{2^d\pi^{d-1}}{(k+1)_{d-1}}$.
\qed

\subsection{Proof of Theorem~\ref{thm:stability}}
For the proof of estimates (\ref{B})-(\ref{BI}) we first note that by (\ref{needlet-estim})
\[
\sum_{\xi \in \chi_j}\|h_{j,\xi} \|^p_p
\leq c 2^{jd p/2}2^{- jd }
\sum_{\xi \in \chi_j} \frac 1{\big(2^{-j} + \sqrt{1-|\xi|^2}\big)^{p/2-1}}
\]
and we need an upper bound for
$\Omega_r := 2^{-jd}\sum_{\xi \in \chi_j}  \frac 1{\big( 2^{-j} + \sqrt{1-|\xi|^2}\big)^r}$.
To this end, we will use the natural bijection between $B^d$ and $\bS^d_+$ considered in
the remark in \S\ref{defAB}.
Thus for $ x \in B^d$ we write
$\tilde{x} =(x, \sqrt{1-|x|^2} )\in \bS^{d}_+$.
Let $\tilde{p}=(0,1)$ be the "north pole" of $\bS^{d}$.
For $\tilde{\xi}\in \chi_j$ we denote by
$B_{\bS^d}(\tilde{\xi}, \rho)$ is the geodesic ball on $\bS^{d}$ of radius $\rho$ centered at $\tilde{\xi}$,
i.e.
$B_{\bS^d}(\tilde{\xi}, \rho):=\{\tilde{x}\in \bS^d: d_{\bS^{d}}(\tilde{x},\tilde{p})< \rho\}$,
where
$d_{\bS^{d}}(\tilde{x},\tilde{p})
=\Arccos \, (\sqrt{1-|x|^2})
=\Arccos \, \langle \tilde{x}, \tilde{p}\rangle$
is the geodesic distance between $\tilde{x}$, $\tilde{p}$.
Using that
$||u|-|\xi||\le |\langle \tilde{u}, \tilde{p}\rangle  - \langle \tilde{\xi}, \tilde{p}\rangle|
\leq  d_{\bS^{d}}(\tilde{\xi},\tilde{u}) \le \rho$ for $\tilde{u}\in B_{\bS^d}(\tilde{\xi}, \rho)$,
and $\rho = \tau 2^{-j} \le 2^{-j}$ (see Proposition~\ref{prop:CUB}), it follows that
\[
\frac 1{  2^{-j} + \sqrt{1-|\xi|^2}}
\leq  \frac 2{2^{-j} + \sqrt{1-| u|^2}}
=\frac 2{2^{-j} +\langle \tilde{u}, \tilde{p}\rangle}
\quad \forall \tilde{u}\in B_{\bS^d}(\tilde{\xi}, \rho).
\]
On the other hand, we have
$| B_{\bS^d}(\tilde{\xi}, \rho) |
= |\bS^{d-1}|\int_0^\rho (\sin \theta)^{d-1} d\theta
\geq \rho^d |\bS^{d-1}|\frac{2^{d-1}}{d\pi^{d-1}}$
with $|\bS^{d-1}| =\frac{2\pi^{d/2}}{\Gamma(d/2)}$.
We use the above and the fact that the balls $\{B_{\bS^d}(\tilde{\xi}, \rho)\}_{\xi\in\chi_j}$
are disjoint to obtain
\begin{align*}
\Omega_r
&\le 2^{-jd}\sum_{\xi \in \chi_j}
\frac 1{|B_{\bS^d}(\tilde{\xi}, \rho)|}\int_{B_{\bS^d}(\tilde{\xi},\rho)}
\frac 1{(2^{-j} +\langle \tilde{u}, \tilde{p}\rangle)^r}d\sigma(\tilde{u})\\
& \le c \int_{\bS^d_+}
\frac 1{ (2^{-j} +\langle \tilde{u}, \tilde{p}\rangle)^r}d\sigma(\tilde{u})
\le c|\bS^{d-1}|\int_0^{\pi/2} \frac{(\sin \theta)^{d-1}}{(2^{-j}+\cos\theta)^r} d\theta\\
& \le c\int_0^{\pi/2} \frac{\sin \theta}{(2^{-j}+\cos\theta)^r} d\theta
= c\int_0^1 \frac{1}{(2^{-j}+t)^r} dt
\le c(d, \tau, r)\int_{2^{-j}}^2 t^{-r}dt.
\end{align*}
This yields estimates (\ref{B})-(\ref{BI}).


We now turn to the proof of estimate~(\ref{B2prime}).
We will employ the maximal operator $\cM_t$ ($t>0$), defined by
\begin{equation}\label{def: max-op}
\cM_tf(x):=\sup_{B\ni x}\left(\frac1{|B|}\int_B|f(y)|^t dy\right)^{1/t},
\quad x\in B^d,
\end{equation}
where the sup is over all balls  $B\subset B^d$ with respect to the distance $d(\cdot, \cdot)$
from (\ref{def-distance}) containing~$x$.
It is easy to show that (see \S2.3 in \cite{pxukball})
the Lebesgue measure on $B^d$ is a doubling measure with respect to
the distance $d(\cdot, \cdot)$. Hence the general theory of maximal
operators applies. In particular, the Fefferman-Stein vector-valued
maximal inequality is valid:
If $0<p<\infty, 0<q\le\infty$, and
$0<t<\min\{p,q\}$ then for any sequence of functions
$\{f_\nu\}_\nu$ on $B^d$
\begin{equation}\label{max-ineq}
\|(\sum_{\nu=1}^\infty|\cM_tf_\nu(\cdot)|^q)^{1/q}\|_p
\le c\|(\sum_{\nu=1}^\infty|
f_\nu(\cdot)|^q)^{1/q}\|_p.
\end{equation}
Denote by $B(\xi, r)$ the projection of $B_{\bS^d}(\tilde{\xi},r)$ onto $B^d$,
i.e. $B(\xi, r):=\{x\in B^d: d(x, \xi)<r\}$.
By \cite[Lemma 2.5]{pxukball}, we have
\begin{equation}\label{est-Max-B}
(\cM_t \ONE_{B(\xi, r)})(x) \ge c\Big(1+\frac{d(\xi , x)}{r}\Big)^{-(d+1)/t},
\quad \xi\in B^d, \; 0<r\le \pi.
\end{equation}
It is easy to see (cf. \cite{pxukball}) that
\begin{equation}\label{maesure-ball}
|B(\xi, \rho)| \sim 2^{-jd}(2^{-j}+\sqrt{1-|\xi|^2})\sim 2^{-jd}W_j(\xi),
\quad \xi\in \chi_j.
\end{equation}
Also, we let $\tONE_{E}:= \frac{1}{|E|}\ONE_{E} $ denote the $L^2$-normalized characteristic
function of $E\subset B^d$.
Then (\ref{needlet-local}) and (\ref{maesure-ball}) imply
\begin{equation}\label{norm-hj}
\|h_{j, \xi}\|_p \sim \|\tONE_{B(\xi, \rho)}\|_p, \quad \xi\in \chi_j.
\end{equation}
Now, pick $0<t<1$ and $M > (d+1)/t$. From (\ref{needlet-local}) and (\ref{est-Max-B})
it follows that
\begin{equation}\label{max1}
|h_{j, \xi}(x)| \le c(\cM_t \tONE_{B(\xi, \rho)})(x), \quad x\in B^d.
\end{equation}
Using this, the maximal inequality (\ref{max-ineq}), and (\ref{norm-hj}) we obtain
\begin{eqnarray*}
\|\sum_{\xi\in \chi_j} d_\xi h_{j, \xi}\|_p
&\le& c\|\sum_{\xi\in \chi_j} \cM_t(d_\xi \tONE_{B(\xi, \rho)})\|_p
\le c\|\sum_{\xi\in \chi_j} d_\xi\tONE_{B(\xi, \rho)}\|_p\\
&\le& c\Big(\sum_{\xi\in \chi_j} \|d_\xi\tONE_{B(\xi, \rho)}\|_p^p\Big)^{1/p}
\le c \Big(\sum_{\xi\in \chi_j}\| d_\xi h_{j, \xi}\|_p^p\Big)^{1/p}.
\end{eqnarray*}
This completes the proof of (\ref{B2prime}). \qed

\bibliographystyle{plain}
\bibliography{pencholin}

\end{document}